\documentclass[11pt]{amsart}

\usepackage{amsmath,amssymb,amscd,amsthm,amsxtra,amsfonts}
\usepackage[dvips]{graphics,epsfig}
\usepackage[all]{xy}
\usepackage{url}

\usepackage{color}

\allowdisplaybreaks[2]

\usepackage{color}
\definecolor{orange}{rgb}{1,0.5,0}
\definecolor{darkgreen}{rgb}{.2,0.4,0.3}

\newtheorem{theorem}{Theorem} [section]

\newtheorem{lemma}[theorem]{Lemma}
\newtheorem{proposition}[theorem]{Proposition}
\newtheorem{remark}[theorem]{Remark} 

\newtheorem{definition}{Definition} [section]
\newtheorem{corollary}[theorem]{Corollary}

\DeclareMathOperator*{\supp}{supp}

\newcommand{\noi}{\noindent}
\newcommand{\N}{\mathbb{N}}
\newcommand{\M}{\mathcal{M}}

\newcommand{\R}{\mathbb{R}}
\newcommand{\T}{\mathbb{T}}
\newcommand{\C}{\mathbb{C}}

\newcommand{\eps}{\varepsilon}

\newcommand{\ld}{\lambda}
\newcommand{\Ld}{\Lambda}

\newcommand{\ft}{\widehat}

\newcommand{\cj}{\overline}
\newcommand{\dx}{\partial_x}
\newcommand{\dt}{\partial_t}

\newcommand{\wk}{\rightharpoonup}

\numberwithin{equation}{section}
\numberwithin{theorem}{section}

\begin{document}

\title
[Explicit formula for the solution of the
Szeg\"o equation on the real line]
{\bf Explicit formula for the solution of the
Szeg\"o equation on the real line
and applications}

\author{Oana Pocovnicu}

\address{Oana Pocovnicu\\
Laboratoire de Math\'ematiques d'Orsay\\
Universit\'e Paris-Sud (XI)\\
91405, Orsay Cedex, France}

\email{oana.pocovnicu@math.u-psud.fr}



\begin{abstract}
We consider the cubic Szeg\"{o} equation
\begin{equation*}
i\dt u=\Pi(|u|^{2}u)
\end{equation*}

\noi
in the Hardy space $L^2_+(\R)$ on the upper half-plane,
where $\Pi$ is the Szeg\"o projector.
It is a model for totally non-dispersive evolution equations
and is completely integrable in the sense that it admits a Lax pair.
We find an explicit formula for solutions of the Szeg\"o equation.
As an application,
we prove soliton resolution in $H^s$ for all $s\geq 0$, for generic data.
As for non-generic data, we construct an example for which
soliton resolution holds only in $H^s$, $0\leq s<1/2$,
while the high Sobolev norms grow to infinity over time,
i.e. $\lim_{t\to\pm\infty}\|u(t)\|_{H^s}=\infty,$ $s>1/2.$
As a second application, we construct explicit generalized action-angle coordinates
by solving the inverse problem for the Hankel operator $H_u$ appearing in the Lax pair.
In particular, we show that the trajectories
of the Szeg\"o equation with generic data
are spirals around Lagrangian
toroidal cylinders $\T^N\times\R^N$.
\end{abstract}

\date{\today}
\maketitle

\tableofcontents

\section{Introduction}

\subsection{Cubic Szeg\"o equation}
One of the most important properties
in the study of the nonlinear Schr\"odinger equations (NLS) is {\it dispersion}.
It is often exhibited in the form of the Strichartz estimates
of the corresponding linear flow.
In case of the cubic NLS:
\begin{equation}\label{eqn: Schrodinger}
i\partial_t u+\Delta u=|u|^2u, \quad (t,x)\in\R\times M,
\end{equation}

\noi
G\'erard and Grellier \cite{PGSGX} remarked that
there is a lack of dispersion
when $M$ is a sub-Riemannian manifold
(for example, the Heisenberg group).
In this situation, many of the classical arguments
used in the study of NLS no longer hold.
As a consequence, even the problem of global well-posedness of \eqref{eqn: Schrodinger}
on a sub-Riemannian manifold still remains open.
In \cite{PGSG, PGSGX},
G\'erard and Grellier introduced a model of a non-dispersive Hamiltonian equation
called {\it the cubic Sz\"ego equation.} (See \eqref{eq:szego} below.)
The study of this equation is expected to give new tools to be used in understanding
existence and other properties of smooth solutions of NLS in the absence of dispersion.

In this paper we will consider the Szeg\"o equation on the real line. The space
of solutions in this case is the Hardy space $L^2_+(\R)$
on the upper half-plane $\C_+ = \{ z; \text{Im} z > 0\}$,
defined by
\[L^2_+(\R)=\Big\{f \text{ holomorphic on } \C_+;
\|g\|_{L^2_+(\R)}:=\sup_{y>0}\bigg(\int_{\R}|g(x+iy)|^{2}dx\bigg)^{1/2}<\infty  \Big\}.\]

\noi
In view of the Paley-Wiener theorem, we identify
this space of holomorphic functions on $\C_+$
with the space of their boundary values:
\[L^2_+(\R)=\{f\in L^2(\R); \, \supp{\hat{f}}\subset [0,\infty)\}.\]

\noi
The corresponding Sobolev spaces $H^s_+(\R)$, $s\geq 0$ are defined by:
\begin{align*}
H^s_+(\R)=&\big\{h\in L^2_+(\R);  \|h\|_{H^s_+}:
=\bigg{(}\frac{1}{2\pi}\int_0^{\infty}(1+|\xi|^2)^{s}|\hat{h}(\xi)|^2d\xi\bigg{)}^{1/2}<\infty\big\}.
\end{align*}
Similarly, we define the homogeneous Sobolev norm for $h\in \dot{H}^s_+$ by
\begin{align*}
||h\|_{\dot{H}^s_+}:=\bigg{(}\frac{1}{2\pi}\int_0^{\infty}|\xi|^{2s}|\hat{h}(\xi)|^2\bigg{)}^{1/2}<\infty.
\end{align*}

\noi
Endowing $L^2(\R)$ with the usual scalar product $(u,v)=\int_{\R}u\bar{v}$, we define the Szeg\"{o} projector
$\Pi:L^{2}(\R)\to L^2_+(\R)$ to be the projector onto the non-negative frequencies,
\[\Pi (f)(x)=\frac{1}{2\pi}\int_{0}^{\infty}e^{ix\xi}\hat{f}(\xi)d\xi.\]

\noi
For $u \in L^2_+(\R)$,
we consider \textit{the Sz\"ego equation on the real line}:
\begin{equation}\label{eq:szego}
i\dt u=\Pi(|u|^{2}u), \quad (t, x) \in \R\times \R.
\end{equation}

\noi
Endowing $L^2_+$ with the symplectic structure $\omega(u,v)=4\textup{Im}\int_{\R}u\bar{v}$,
we have that the Szeg\"o equation is a Hamiltonian evolution associated to the Hamiltonian
\begin{equation*}
E(u)=\int_{\R}|u|^4 dx
\end{equation*}

\noi
defined on $L^4_+(\R)$.
From this structure, we obtain the formal conservation law of the energy $E(u(t))=E(u(0))$.
The invariance under translations and under modulations provides two more conservation laws,
the mass $Q(u(t))=Q(u(0))$ and the momentum $M(u(t))=M(u(0))$, where
\begin{equation*}
Q(u)=\int_{\R}|u|^2 dx\quad \text{and} \quad M(u)=\int_{\R}\bar{u}Du\,dx,\
\text{ with } D = -i \dx.
\end{equation*}

\noi
Noting that $Q(u)+M(u)=\|u\|_{H^{1/2}_+}^2$,
we have $\|u(t)\|_{H^{1/2}_+}=\|u(0)\|_{H^{1/2}_+}$.
Hence, $H^{1/2}_+$ is the natural space for studying the well-posedness of the equation.
In \cite[Theorem 1.1]{pocov}, it was shown that the Szeg\"o equation
on the real line is globally well-posed in $H^{1/2}_+(\R)$
and satisfies the persistence of regularity,
i.e. if $u_0\in H^s_+(\R)$ for some $s>\frac{1}{2}$, then $u\in C(\R,H^{s}_+(\R))$.

\medskip

First of all we recall some notions and properties concerning the Szeg\"o equation.
We refer the readers to \cite{pocov}
for more details.
The main property of the Szeg\"o equation
 is that it is completely integrable in the sense that it
possesses a Lax pair structure \cite[Proposition 1.4]{pocov}.
We first define two important classes of operators on $L^2_+$,
{\it the Hankel and Toeplitz operators}.
The Lax pair is given in terms of these operators
in Proposition \ref{th:Lax pair}.

A Hankel operator $H_{u}:L^2_+\to L^2_+$ of symbol $u\in H^{1/2}_+$ is defined by
\[H_{u}(h)=\Pi(u\bar{h}).\]

\noi
Then, as it was shown in Lemma 3.5 in \cite{pocov},
 $H_u$ is Hilbert-Schmidt and $\C$-antilinear. Moreover, it satisfies
 the following identity:
\begin{equation}\label{sym H_u}
(H_{u}(h_{1}),h_{2})=(H_{u}(h_{2}),h_{1}).
\end{equation}

\noi
As a consequence, $H_{u}^2$ is a self-adjoint linear operator.
A Toeplitz operator $T_{b}:L^2_+\to L^2_+$ of symbol $b\in L^{\infty}(\R)$ is defined by
\[T_{b}(h)=\Pi(bh).\]

\noi
Then, $T_{b}$ is $\C$-linear and bounded.  Moreover, $T_b$ is self-adjoint if and only if $b$ is real-valued.

\begin{proposition}[Proposition 1.5 in \cite{pocov}]\label{th:Lax pair}
Let $u\in C(\R;H_+^{s})$ for some $s>\frac{1}{2}$.
The cubic Szeg\"{o} equation
\eqref{eq:szego}
is equivalent to
the following evolution equation:
\begin{equation*}
\frac{d}{dt}H_{u}=[B_{u},H_{u}],
\end{equation*}

\noi
where
\begin{equation}\label{eqn:Bu}
B_{u}=\frac{i}{2}H_{u}^{2}-iT_{|u|^{2}}.
\end{equation}

\noi
In other words, the pair $(H_u,B_u)$ is a
Lax pair for the cubic Szeg\"{o} equation on the real line.
\end{proposition}

According to the classical theory developed by Lax \cite{Lax}, a direct
consequence of the above proposition is the following corollary:

\begin{corollary}
Let $U(t)$ be an operator on $H^{1/2}_+$ defined by:
\begin{equation}\label{eqn:d_tB_u}
\frac{d}{dt}U(t)=B_{u(t)}U(t), \,\,\,\,\,\,\,\,\,\,\,\,\,\, U(0)=I.
\end{equation}
Then, $U(t)$ is a unitary operator and if $u$ is a solution
of the Szeg\"o equation
\eqref{eq:szego}
with initial condition $u_0$, we have:
\begin{equation}\label{eqn:H_u}
H_{u(t)}=U(t)H_{u_0}U(t)^{\ast}.
\end{equation}

\noi
This yields
\begin{equation}\label{eqn:URan}
U(t)\big(\textup{Ker}(H_{u_0})\big)\subset \textup{Ker} (H_{u(t)}),\,\,\,\,\,\,\,\,\,\,\
U(t)\big{(}\textup{Ran}(H_{u_0})\big{)}\subset \textup{Ran} (H_{u(t)}).
\end{equation}
\end{corollary}

Another consequence of the Lax pair structure is the existence of an infinite sequence of
conservation laws. More precisely, the following corollary holds.
\begin{corollary}\label{cor:conserv}
Define $J_{n}(u):=(u,H_u^{n-2}u)$ for all $n\geq 2$. Then
$J_{2k}(u)$, $k\in\N^{\ast}$, are conserved quantities for the Szeg\"o equation.
In particular, $J_2(u)=Q(u)$, $J_4(u)=\frac{E(u)}{2}$, and
we recover the conservation laws of the mass and energy.
\end{corollary}

\begin{remark}\label{conservation laws} \rm
Using the Mikhlin multiplier theorem, we can prove that $J_{2k}(u)\leq \|u\|_{L^{2k}}^{2k}$.
Then, by the Sobolev embedding we have that $J_{2k}(u)\leq \|u\|_{H^{1/2}_+}^{2k}$.
This shows that the strongest conservation law for the Szeg\"o equation is the $H^{1/2}_+$-norm.
\end{remark}

\subsection{Main results}
It turns out that rational functions play an important role
in studying the Hankel operators, and thus the Szeg\"o equation.
In the following, we first consider solutions for the
Szeg\"o equation with rational function
initial data
 $u_0\in\M(N)$, where $\M(N)$
is defined below.

\begin{definition}
Let $N\in\N^{\ast}$. We denote by $\M(N)$ the set of rational functions of the form
$$\frac{A(z)}{B(z)},$$
where $A\in\C_{N-1}[z]$, $B\in\C_{N}[z]$, $0\leq \deg(A)\leq N-1$, $\deg(B)=N$, $B(0)=1$, $B(z)\neq 0$,
for all $z\in\C_{+}\cup\R$, and $A$ and $B$ have no common factors.
\end{definition}

\noi
Note that $\M(N)$ is a $4N$-dimensional real manifold, $\M(N)\subset H^{s}_+(\R)$
for all $s\geq 0$,
and that $\bigcup_{N=1}^{\infty}\M(N)$ is dense in $L^2_+$ \cite[Lemma 6.2.1]{Nikolskii}.
Moreover, they remain invariant under the flow.

\begin{proposition} \label{PROP:invariant}
The manifolds $\M(N)$ are invariant under the flow of the Szeg\"o equation.
\end{proposition}

In order to prove this statement, we recall a Kronecker-type theorem.

\begin{proposition}[Theorem 2.1 in \cite{pocov}]
\label{PROP:kronecker}
Let $u\in H^{\frac{1}{2}}_+$. Then $u\in\M(N)$ if and
only if $\textup{rk}(H_u)=N$.
Moreover, if $u=\frac{A}{B}\in\M(N)$, where $A$ and $B$ are relatively prime, $B(0)=1$,
$B(x)=(x-p_1)^{m_1}\dots(x-p_k)^{m_k}$, $m_1+\dots m_k=N$,
and $\textup{Im}(p_j)<0$ for all $j=1,2,\dots,k$
then, we have that
\begin{equation}\label{eq:ranH_u}
\textup{Ran}(H_u)=\textup{span}_{\C}\bigg\{
\frac{1}{(x-p_j)^{l_j}}, j=1,2,\dots,k \text{ and } l_j=1,2,\dots,m_j\bigg\}.
\end{equation}

\end{proposition}

\begin{proof}[Proof of Proposition \ref{PROP:invariant}]
By equation \eqref{eqn:H_u} and Proposition \ref{PROP:kronecker},
we have that if $u_0\in\M(N)$,
then $\textup{rk}(H_{u(t)})=\textup{rk}(H_{u_0})=N$.
Thus the corresponding solution $u(t)\in\M(N)$ for all $t\in\R$.
\end{proof}

As a corollary  of the Kronecker-type theorem \cite[Remark 2.2]{pocov},
we also have that if $u\in\M(N)$
then $u\in\text{Ran}(H_u)$,
i.e. there exists a unique element
$g\in\text{Ran}(H_{u})$ such that
\begin{equation}\label{eqn: def g}
u=H_{u}(g).
\end{equation}

\noi
This yields $\Pi (u(1-\bar{g}))=0$, which gives:
\begin{align}\label{eqn:bar{u}(1-g)}
\bar{u}(1-g)\in L^2_+.
\end{align}

\noi
An important property of Hankel operators, that will be a key point in this paper,
is their characterization using
the shift operators $\tilde{T}_{\ld}:L^2_+\to L^2_+$, $\ld>0$,
\begin{equation*}
\tilde{T}_{\ld}f(x)=e^{i\ld x}f(x).
\end{equation*}

\noi
More precisely, the bounded operator $H:L^2_+\to L^2_+$
is a Hankel operator if and only if
\begin{equation}\label{tilde T_ld}
\tilde{T}_{\ld}^{\ast}H=H\tilde{T}_{\ld}
\end{equation}

\noi
for all $\ld>0$, \cite[p. 273]{Nikolskii}. The adjoint
$\tilde{T}_{\ld}^{\ast}:L^2_+\to L^2_+$, defined by
\begin{equation*}
\tilde{T}_{\ld}^{\ast}f(x)=e^{-i\ld x}(f\ast\mathcal{F}^{-1}(\chi_{[\ld,\infty)}))(x),
\end{equation*}

\noi
is very inconvenient to use. Then, for rational functions $u$,
we define the infinitesimal shift operator
$T:\textup{Ran}(H_u)\to \textup{Ran}(H_u)$,
\begin{equation} \label{eq: def T}
Tf(x)=xf(x)-\big(\lim_{x\to\infty}xf(x)\big)(1-g(x))
\end{equation}

\noi
and prove that
\begin{equation}\label{comute T, H}
T^{\ast}H_u=H_uT.
\end{equation}

In the general case, when $u$ is not a rational function,
$u$ does not always belong to $\textup{Ran}(H_u)$.
Thus, $g$ satisfying \eqref{eqn: def g} does not always exist.
If such $g$ does not exist,
the above definition \eqref{eq: def T} of $T$ does not make sense.
We then propose, in Section 3, to extend a definition for $T^{\ast}$
(see \eqref{eq:def T star} below)
and pursue our work using $T^{\ast}$ rather than $T$.

\medskip

Next, we recall the definition
and the characterization of soliton solutions for the Szeg\"o equation.
See \cite{pocov} for details.
\begin{definition} \rm
A {\it soliton} for the Szeg\"o equation on the real line is a solution
$u$ with the property that there exist $c,\omega\in\R$, $c\neq 0$ such that
\[u(t,x)=e^{-it\omega}u_0(x-ct).\]
\end{definition}

\noi
In \cite[Theorem 2]{pocov}
it was proved that all the solitons for the Szeg\"o equation on $\R$ are of the form
\begin{equation}\label{soliton}
u(t,x)=e^{-i\omega t}\phi_{C,p}(x-ct),
\end{equation}

\noi
where $\phi_{C,p}=\frac{C}{x-p}$, $\omega=\frac{|C|^2}{4(\textup{Im}p)^2}$, $c=\frac{|C|^2}{-2\textup{Im}p}$,
 $C,p\in\C$, and $\textup{Im}p<0$. Hence, a soliton
of the Szeg\"o equation on $\R$ is a simple fraction $u(t,x)=\frac{C e^{-i\omega t}}{x-ct-p}\in\M(1)$,
where $\textup{Im}(p)<0$.

\medskip

We are now ready to state the main results of this paper. In the first place we find an explicit
formula for the solutions of the Szeg\"o equation with rational function data.

\begin{theorem}[Explicit formula in the case of rational function data]\label{thm:general formula}
Suppose that $u_0\in\M(N)$ and $H^2_{u_0}$ has positive eigenvalues $\ld_1^2\leq \ld_2^2\leq\dots\leq\ld_N^2$.
We will assume that $\ld_j>0$ for all $j=1,2,\dots,N$.
Choose a complex orthonormal basis
$\{e_j\}_{j=1}^N$
of $\textup{Ran}(H_{u_0})$,
consisting of  eigenvectors of $H_{u_0}^2$
such that $H_{u_0}e_j=\ld_je_j$ for all $j=1,2,\dots,N$.
Let $W(t)=e^{i\frac{t}{2}H_{u_0}^2}$ and $\beta_j=(g_0,e_j)$.

We define an operator $S(t)$ on $\textup{Ran}(H_{u_0})$ in the following way.
Fix $j \in \{1, \dots, N\}$,
and let $\ld_j^2$ be an eigenvalue of multiplicity $m_j$.
Moreover, let $M_j\subset\N$ be the set of all indices $k$ such that $H_{u_0}e_{k}=\ld_j e_{k}$.
Then, $S(t)$ in the basis $\{e_j\}_{j=1}^N$ is defined by the matrix
\begin{align}\label{eqn:S}
S(t)_{k,j}=
\begin{cases}
\frac{\ld_j}{2\pi i(\ld_k^2-\ld_j^2)}
\Big(\ld_j e^{i\frac{t}{2}(\ld_k^2-\ld_j^2)}\cj{\beta}_j\beta_k
-\ld_k e^{i\frac{t}{2}(\ld_j^2-\ld_k^2)}\beta_j\cj{\beta}_k\Big),\text{ if } k\in\{1,\dots,N\}\setminus M_j,\\
\frac{\ld_j^2}{2\pi}\cj{\beta}_j\beta_k t+(Te_j,e_k), \text{ if } k\in M_j.
\end{cases}
\end{align}

Then, we have the following explicit formula for the solution of the Szeg\"o equation:
\begin{equation*}
u(t,x)=\frac{i}{2\pi}\Big(u_0,W(t)(S- xI)^{-1}W(t)g_0\Big), \textup{ for all } \quad (t, x)\in\R \times \R.
\end{equation*}

\end{theorem}

We extend the explicit formula to more general initial data, that are not necessarily rational functions,
 in the following corollary.
\begin{corollary}[A first generalization of the explicit formula]\label{cor:gen case}
Let $u_0\in H^{1/2}_+$ be a general initial condition.
Denote by $\{\ld_j^2\}_{j=1}^{\infty}$ the positive eigenvalues of the operator $H_{u_0}^2$.
We assume that $\ld_j>0$ for all $j\in\N$.
Choose a complex orthonormal basis $\{e_j\}_{j=1}^{\infty}$ of $\textup{Ran}(H_{u_0})$
consisting of eigenvectors of $H_{u_0}^2$
such that $H_{u_0}e_j=\ld_je_j$ for all $j\in \mathbb{N}^{\ast}$.
Denote $W(t)=e^{i\frac{t}{2}H_{u_0}^2}$ and $\beta_j=\frac{1}{\ld_j}(e_j,u_0)$.

We define an operator $S(t)$
on $\textup{Ran}(H_{u_0})$ in the following way.
Fix $j \in \mathbb{N}^*$,
and let $\ld_j^2$ be an eigenvalue of multiplicity $m_j$.
Moreover, let $M_j\subset\mathbb{N}^*$ be the set of all indices $k$ such that $H_{u_0}e_{k}=\ld_j e_{k}$.
Then, $S(t)$ is defined by
\begin{align*}
\big(S(t)e_j,e_k\big)=
\begin{cases}
\frac{\ld_j}{2\pi i(\ld_k^2-\ld_j^2)}
\Big(\ld_j e^{i\frac{t}{2}(\ld_k^2-\ld_j^2)}\cj{\beta}_j\beta_k
-\ld_k e^{i\frac{t}{2}(\ld_j^2-\ld_k^2)}\beta_j\cj{\beta}_k\Big), \text{ if } k\in\N\setminus M_j,\\
\frac{\ld_j^2}{2\pi}\cj{\beta}_j\beta_k t+(Te_j,e_k), \text{ if } k\in M_j.
\end{cases}
\end{align*}

\noi
Denote by $\bar{S}$ the closure of the operator $S$.

If the sequence $\{\beta_j\}_{j\in\N}$ is in $\ell^2$,
then there exists $g_0\in\textup{Ran}(H_{u_0})$ such that
$u_0=H_{u_0}(g_0)$. Moreover, for $\textup{Im}z>0$, the following formula for the
solution of the Szeg\"o equation with initial condition $u_0$ holds:
\begin{equation*}
u(t,z)=\frac{i}{2\pi}\Big(u_0,W(t)(\bar{S}- \bar{z}I)^{-1}W(t)g_0\Big).
\end{equation*}
\end{corollary}

The condition $\{\beta_j\}_{j\in\N}\in\ell^2$ characterizes all initial data satisfying
$u_0\in\textup{Ran}H_{u_0}$. In particular, by \eqref{eqn: def g}, it is satisfied by all rational functions.
However, simple non-rational functions, like $\frac{e^{i\alpha x}}{x+i}$
with $\alpha>0$, do not satisfy it,
and hence Corollary \ref{cor:gen case} is not applicable.
In the following theorem, we extend the explicit formula to even more general initial data.

\begin{theorem}[Explicit formula for general data]\label{general case}
Let $u_0\in H^s_+$, $s>\frac{1}{2}$, $xu_0\in L^{\infty}(\R)$.
With the notations in Corollary \ref{cor:gen case}, we define an operator
$S^{\ast}(t)$ on $\textup{Ran}(H_{u_0})$ in the following way.
Fix $j \in \mathbb{N}^*$. If $\ld_j^2$ is an eigenvalue of multiplicity $m_j$ and
$M_j\subset\N$ is the set of all indices $k$ such that $H_{u_0}e_{k}=\ld_j e_{k}$,
then
\begin{align*}
(S^{\ast}(t)e_j,e_k)=
\begin{cases}
\frac{\ld_k}{2\pi i(\ld_k^2-\ld_j^2)}
\Big(\ld_k e^{i\frac{t}{2}(\ld_k^2-\ld_j^2)}\cj{\beta}_j\beta_k
-\ld_j e^{i\frac{t}{2}(\ld_j^2-\ld_k^2)}\beta_j\cj{\beta}_k\Big), \text{ if } k\in\N\setminus M_j, \\
\frac{\ld_k^2\cj{\beta}_j\beta_k}{2\pi}t+(T^{\ast}e_j,e_k), \text{ if } k\in M_j.
\end{cases}
\end{align*}

\noi
Let A be the closure of $S^{\ast}$.
Then, for $\textup{Im}z>0$, the solution of the Szeg\"o equation writes
\begin{align*}
u(t,z)=\lim_{\eps\to 0}\frac{i}{2\pi }\Big(W^{\ast}(t)(A-zI)^{-1}W^{\ast}(t)u_0,\frac{1}{1-i\eps z}\Big).
\end{align*}

\noi
Let $S^{\ast}_{\ld}$ be the semi-group of contractions
whose infinitesimal generator is $-iA$.
Then, the above formula is equivalent to
\begin{align*}
\ft{u}(t,\ld)=\lim_{\eps\to 0}\frac{1}{2\pi}\Big(W^{\ast}(t)S^{\ast}_{\ld}(t)W^{\ast}(t)u_0,\frac{1}{1-i\eps x}\Big)
, \quad \text{ a.e. } \ld\in\R.
\end{align*}

\end{theorem}

\medskip

\begin{definition} \rm
A function $u_0\in \M(N)$ is called {\it generic} if the operator
$H_{u_0}^2$ has simple eigenvalues $0<\ld_1^2<\ld_2^2<\dots<\ld_N^2$ and $|(u_0,e_j)|\neq 0$,
for all $j=1,2,\dots,N$. We denote by $\M(N)_{\textup{gen}}$
the set of generic rational functions in $\M(N)$.

A function $u_0$ is called {\it strongly generic} if it is generic and, in addition, $|(u_0,e_j)|\neq |(u_0,e_k)|$ for all $k\neq j$.
We denote by $\M(N)_{\textup{sgen}}$
the set of strongly generic rational functions in $\M(N)$.
\end{definition}

\noi
The sets $\M(N)_{\textup{gen}}$ and $\M(N)_{\textup{sgen}}$ are indeed generic, in the sense that they are
open, dense subsets of $\M(N)$. 
As in \cite[Theorem 7.1]{PGSG}, we have that 
$\det \big(J_{2(m+n)}\big)_{1\leq m,n\leq N}\neq 0$
if and only if $H_{u}^{2k}(g)$, $k=1,2,\dots,n$, are linearly independent.
Decomposing $g,H^2g,\dots,H^{2(N-1)}g$
in the basis $\{e_j\}_{j=1}^N$,
we obtain that the determinant of the matrix
which contains these vectors as columns is:
\begin{align*}
\left\vert
\begin{matrix}
\nu_1 & \ld_1^2\nu_1 & \dots & \ld_1^{2(N-1)}\nu_1\\
\nu_2 & \ld_2^2\nu_2 & \dots & \ld_2^{2(N-1)}\nu_2\\
\vdots & \vdots & \ddots & \vdots\\
\nu_N & \ld_N^2\nu_N & \dots & \ld_N^{2(N-1)}\nu_N
\end{matrix}
\right\vert
=\nu_1\dots\nu_N
\left\vert
\begin{matrix}
1 & \ld_1^2 & \dots & \ld_1^{2(N-1)}\\
1 & \ld_2^2 & \dots & \ld_2^{2(N-1)}\\
\vdots & \vdots & \ddots & \vdots\\
1 & \ld_N^2 & \dots & \ld_N^{2(N-1)}
\end{matrix}
\right\vert,
\end{align*}

\noi
where $\nu_j:=\frac{1}{\ld_j}|(u,e_j)|$. Thus, the fact that
 $g,H^2g,\dots,H^{2(N-1)}g$ are linearly independent is equivalent to
$(u,e_j)\neq 0$, $j=1,2,\dots,N$ and $\ld_j$ are all distinct. Therefore,
\[\M(N)_{\textup{gen}}=\big\{u_0\in\M(N)\big|\det \big(J_{2(m+n)}\big)_{1\leq m,n\leq N}\neq 0\big\}\]

\noi
is an open, dense subset of $\M(N)$. By Theorem \ref{thm:action-angle} below, we obtain that $\chi:\M(N)_{\textup{gen}}\to \Omega$
(see \eqref{eq:chi} blow)
is a diffeomorphism. Since $\M(N)_{\textup{sgen}}$ corresponds, through $\chi$, to an open dense subset of $\Omega$, it results that
$\M(N)_{\textup{sgen}}$ is also generic.

\begin{definition} \rm
We say that {\it soliton resolution} holds in $H^s$ for a solution $u(t)$ of the Szeg\"o equation, if $u(t)$
can be written
as the sum of a finite number of solitons and a remainder $\eps(t,x)$ with the property that $\lim_{t\to\pm\infty}\|\eps(t,x)\|_{H^s}=0$.
\end{definition}

Using the above explicit formula for the solution, we prove the following result:

\begin{theorem}[Solition resolution for strongly generic data]\label{thm:soliton emergence}
Let $u_0\in\M(N)_{\textup{sgen}}$ be a strongly generic initial data for the Szeg\"o equation.
Then, the corresponding solution satisfies the property of soliton resolution in $H^s$ for all $s\geq 0$.
More precisely, with the notations in Theorem \ref{thm:general formula}, we have
\begin{align*}
u(t,x)=\sum _{j=1}^{N}e^{-it\ld_j^2}\phi_{C_j,p_j}(x-\tfrac{\ld_j^2\nu_j^2}{2\pi}t)+\eps(t,x),
\end{align*}

\noi
where $C_j=\frac{i\ld_j\bar{\beta}_j^2}{2\pi }$, $p_j=\textup{Re}(c_j(0))-i\frac{\nu_j^2}{4\pi}$,
and $\lim_{t\to\pm\infty}\|\eps(t,x)\|_{H^s_+}=0$ for all $s\geq 0$.
\end{theorem}

\medskip

Studying the case of non-generic initial data $u_0\in\M(2)$, such that
$H_{u_0}^2$ has a double eigenvalue $\ld_1^2=\ld_2^2$,
we can prove that the soliton resolution holds in $H^s$
only for $0\leq s<1/2$.
It turns out that the $H^s$-norms with $s>1/2$ of such non-generic solutions
grow to $\infty$ as $t\to\pm\infty$.

\begin{theorem}[Partial solition resolution for non-generic data]\label{thm:conterexample}
Let $u_0\in\M(2)$ be such that $H_{u_0}^2$ has a double eigenvalue $\ld^2>0$.
Then the corresponding solution satisfies the property of soliton resolution in $H^s$ for $0\leq s<1/2$.
More precisely,
\begin{align*}
u(t,x)=&e^{-it\ld^2} \phi_{C,p}\big(x-\tfrac{\|u_0\|_{L^2}^2}{2\pi}t\big)
+\eps(t,x),
\end{align*}
\noi
where the first term is a soliton with
$|C|=\frac{\|u_0\|_{L^2}^2}{\sqrt{\pi}\|u_0\|_{\dot{H}^{1/2}}}$,
$\textup{Im}(p)=-\Big(\frac{\|u_0\|_{L^2}}{\|u_0\|_{\dot{H}^1}}\Big)^2$, and
$\eps(t,x)\to 0$ in the all the
$H^s$-norms with $0\leq s<1/2$.

However, $\eps(t,x)$ stays away from zero
and is bounded
in the $L^{\infty}$-norm and $H^{1/2}$-norm.
Moreover, $\lim_{t\to\pm\infty}\|\eps(t,x)\|_{H^s}=\infty$ if $s>1/2$.
\end{theorem}

As a consequence, we obtain the following result:
\begin{corollary}[Growth of high Sobolev norms]\label{Corollary}
The Szeg\"o equation admits solutions $u(t)$ whose high Sobolev norms $H^s$, for $s>1/2$, grow
to infinity:
\[\|u(t)\|_{H^s}\to\infty \textup{ as } t\to\pm\infty.\]

\noi
More precisely, there exists a solution $u$ of the Szeg\"o equation
 and a constant $C>0$ such that
$\|u(t)\|_{H^s}\geq C |t|^{2s-1}$ for sufficiently large $|t|$.
\end{corollary}

\begin{remark} \rm
Corollary \ref{Corollary}
presents an example of solutions
whose high Sobolev norms grow to infinity.
We could observe this
phenomenon by considering
non-generic initial data $u_0$ such that the operator $H_{u_0}^2$
has a double eigenvalue.
We believe that the non-dispersive character of the Szeg\"o
equation plays an important role in the occurrence of this phenomenon.
For example, consider the dispersionless NLS, $iu_t = |u|^2 u$.
Then, $u(t, x) = \phi(x) \exp(-i |\phi(x)|^2t)$
with smooth $\phi$ is a solution, satisfying $\|u(t)\|_{H^s} \sim |t|^s$ for $s \in \mathbb{N}$.
However, the situation is more subtle for the Szeg\"o equation,
 due to the conservation of the $H^{1/2}$-norm.
In particular, this explains why, for the Szeg\"o equation,
only the $H^s$-norms with $s>1/2$ grow to infinity.

Corollary \ref{Corollary}  shows  that the energy is supported on
higher frequencies while the mass is supported on lower frequencies.
This phenomenon is
called ``forward cascade" and  is consistent with some
predictions in the weak turbulence theory.

Previously,
Bourgain constructed, in
\cite{Bourgain1995, Bourgain1995bis, Bourgain1996}, solutions
with Sobolev norms growing to infinity.
He considered, however, Hamiltonian PDEs
involving a spectrally defined Laplacian.
For general (dispersive) Hamiltonian PDEs,
such a phenomenon is not known, but there are several
partial results in this direction.
In \cite[Corollary 5]{PGSG}, G\'erard and Grellier noticed
the growth of Sobolev norms for the Szeg\"o equation on $\T$.
However, their construction of a sequence of solutions $u^{\eps}(t^{\eps})$
whose Sobolev norms become larger depends on the small parameter $\eps$.
In \cite{I team}, Colliander, Keel, Staffilani, Takaoka, and Tao constructed
solutions for the defocusing cubic NLS on $\T^2$ whose high Sobolev norms
become greater than any fixed constant at some time. Kuksin considered in \cite{Kuksin}
the case of small dispersion NLS, $-i\partial_tu+\delta\Delta u=|u|^2u$,\footnote{Note that
this can be considered as a perturbation of the dispersionless NLS. See p.138 in \cite{BourgainParkCity}}
with odd periodic boundary condition, where $\delta$ is a small parameter.
He proved that Sobolev norms of solutions with relatively generic data of unit mass,
grow larger than a negative power of $\delta$. However, these constructions
do not give an example of solution such that
$\sup_{t}\|u(t)\|_{H^s}=\infty$.
\end{remark}

\medskip

In the following theorem
we introduce generalized action-angle coordinates for the Szeg\"o equation in the case of generic rational functions.

\begin{theorem}[Generalized action-angle coordinates]\label{thm:action-angle}
For $u\in\M(N)_{\textup{gen}}$ denote by $0<\ld_1^2<\ld_2^2<\dots\ld_N^2$ the simple positive
eigenvalues of $H_u^2$ and by $\{e_j\}_{j=1}^N$ an orthonormal basis of $\textup{Ran}(H_u)$
such that $H_ue_j=\ld_je_j$. Denote $\nu_j=|(g,e_j)|$, $\phi_j=\arg (g,e_j)$,
and $\gamma_j=\textup{Re}\,(Te_j,e_j)$.

Set $\Omega:=(\R_+^{\ast})^N\times \{0<x_1<x_2\dots<x_N\}\times \T^N\times \R^N$. The application $\chi:\M(N)_{\textup{gen}}\to \Omega$
defined by
\begin{align} \label{eq:chi}
\chi (u)=\Big(\{2\ld_j^2\nu_j^2\}_{j=1}^N,\{4\pi\ld_j^2\}_{j=1}^N,\{2\phi_j\}_{j=1}^N,\{\gamma_j\}_{j=1}^N\Big),
\end{align}

\noi
is a symplectic diffeomorphism.
Moreover, $2\ld_j^2\nu_j^2$, $4\pi\ld_j^2$, $2\phi_j\in \T$, $\gamma_j\in\R$, $j=1,2,\dots,N$ are
generalized action-angle coordinates for the Szeg\"o equation on the real line.
\end{theorem}

As a corollary, we obtain that in the generic case, the trajectories of the Szeg\"o
equation spiral around toroidal-cylinders $\T^N\times\R^N$, $N\in\N^{\ast}$.

\begin{corollary}[Lagrangian toroidal cylinders]\label{cor:toroidal cilinders}
Let $u_0\in\M(N)_{\textup{gen}}$. Consider
\begin{align}\label{TC(u)}
TC(u_0):=\Big\{u\in\M(N)_\textup{gen}|
& H_u^2, H_{u_0}^2 \text{ have same eigenvalues } \ld_j^2 \text{ and same }\nu_j\Big\}.
\end{align}

\noi
Then, $u(t)\in TC(u_0)$ for all $t\in\R$, and the set $TC(u_0)$ is diffeomorphic to a toroidal cylinder
$\T^N\times\R^N$ parameterized by the
coordinates $(2\phi_j,\gamma_j)_{j=1}^{N}$, where $\gamma_j\in\R$, $2\phi_j\in \T$.
\end{corollary}

It seems difficult to extend Theorem \ref{thm:action-angle}
and Corollary \ref{cor:toroidal cilinders} to arbitrary generic functions,
which are not necessarily rational, as we did in Theorem \ref{general case}.
The main reasons are the lack of compactness
and the fact that we are unable to characterize the conditions
$u_0\in H^s_+$, $s>1/2$ and $xu_0(x)\in L^{\infty}(\R)$ in terms of the
spectral data.

\medskip

The present paper was inspired by \cite{PGSGIT}, where G\'erard and Grellier
introduced action-angle coordinates for
the Szeg\"o equation on $\T$. However, \cite{PGSGIT} does not treat the question
of soliton resolution and growth of high Sobolev norms.
Different difficulties are to be overcome in the two settings. In the case of $\R$,
these difficulties are mostly related to the infinitesimal shift operator $T$
in \eqref{eq: def T},
which does not appear in the case of $\T$.

\medskip

\subsection{Structure of the paper}
We conclude this introduction by discussing the structure of the paper
with some details.
In Section 2,
we prove Theorem \ref{thm:general formula},
i.e. find an explicit formula for the solution
of the Szeg\"o equation with rational function initial condition.
 In the case of other completely integrable equations
 like KdV and one dimensional cubic NLS,
 an explicit formula for solutions
was determined by the inverse scattering method
\cite{Ablowitz, Deift, Gardner}.
Since in our case the operator $H_u$ is compact,
we will not apply the inverse scattering method. We find a direct approach
to solve the inverse spectral problem for the Hankel operator $H_u$,
using the Lax pair structure
and the commutation relation \eqref{comute T, H} between
the operator $H_u$ and the infinitesimal shift $T$.

The inverse spectral problem for Hankel operators was considered in several papers,
among which we cite \cite{Abakumov, MPT}. Our results are
more precise than the previous ones and allow us to have a formula
for the symbol $u$ of the Hankel operator $H_u$ only in terms of the spectral data.

Let us describe our strategy in Section 2.
First we notice that
$\hat{u}(\ld)=(u,e^{i\ld x}g)$, $\ld>0$.
Then, we introduce the operators $S_\ld(t)=P_{u_0}U^{\ast}(t)T_{\ld}(t)U(t)$,
$S(t)=U^{\ast}(t)T(t) U(t)$ acting on
 $\textup{Ran}(H_{u_0})$. Exploiting the Lax pair structure,
we obtain that
\begin{equation}\label{formula sol}
u(t,x)=\frac{i}{2\pi}\Big(u_0,W(t)(S-xI)^{-1}W(t)g_0\Big).
\end{equation}

\noi
Since $S$ is defined using $U(t)$ and since the definition of
$U(t)$ \eqref{eqn:d_tB_u} depends on $u(t)$ itself, the above formula is a vicious circle.
To break it, we determine $S$ without using $U(t)$.
The explicit expression for $S$ is obtained by computing the commutator $[H_{u_0}^2,S]$
and the derivative $\frac{d}{dt}S(t)$.

\medskip

In Section 3, we prove Corollary \ref{cor:gen case} and Theorem \ref{general case}.
The proof of Theorem \ref{general case} uses an approximation argument,
based on the remark that $u\in\cj{\textup{Ran}(H_u)}$
for all $u\in H^{1/2}_+$. The crucial step is to define the
``adjoint of the infinitesimal shift operator", $T^{\ast}$,
 for functions which are not necessarily
rational functions (it seems more delicate to define the operator $T$ directly).

Notice that in Theorem \ref{thm:general formula},
$S$ is a matrix whose eigenvalues are not real
and thus the inverse $(S-xI)^{-1}$ can be explicitly computed.
The result obtained in Theorem \ref{general case}
is weaker. The operator $S^{\ast}$ acts between infinite dimensional spaces.
Explicitly computing
$(A-zI)^{-1}$ or the semi-group $S^{\ast}_{\ld}$ comes down to solving
an infinite system of linear differential equations. Therefore, Theorem \ref{general case}
actually states that we can transform our nonlinear infinite dimensional dynamical system into a linear one.

\medskip

In Section 4, we prove Theorem \ref{thm:soliton emergence}.
The soliton
resolution conjecture is believed to be true for many dispersive
equations for which the non-linearity is not strong enough
to create finite-time blow-up.
However, this was proved only for few equations like KdV \cite{KdV}
and one dimensional cubic NLS \cite{Novikov, Manakov}, for which an explicit formula for
the solution is available. For KdV, soliton resolution was proved in $L^{\infty}$
and it was noticed that it is unlikely to hold in $H^1(\R)$
(the remainder may carry a part of the initial energy).
For NLS, soliton resolution was proved in $L^{2}$.
In this case, in addition to solitons, the solution contains
a radiation term, which is a solution of the linear Schrodinger equation.
For both KdV and NLS, the
conjecture holds only for "generic" data.
In Theorem \ref{thm:soliton emergence} we prove that
for strongly generic, rational functions
solutions of the Szeg\"o equation, soliton resolution holds
in all $H^s$, $s\geq 0$.

\medskip

In Section 5, we prove Theorem \ref{thm:conterexample}
and Corollary \ref{Corollary}. We show that soliton resolution still holds,
even for non-generic solutions, but only in $H^s$, $0\leq s<1/2$.
This is probably due to
the fact that $H^{1/2}$ is the space of critical regularity.

The starting point in proving Theorems \ref{thm:soliton emergence} and \ref{thm:conterexample}
 is the explicit formula found in section 2, which we are able to write as a sum of simple
fractions $\frac{C_j(t)}{x-\cj{E}_j(t)}$, $j=1,2,\dots,N$. The key remark is that
the complex conjugates of the poles of $u(t)$, $E_j(t)$, are the eigenvalues of the operator
$T(t)$ acting on $\textup{Ran}(H_{u(t)})$. In the strongly generic case, the eigenvalues of $T(t)$
satisfy $E_j(t)=a_jt+b_j+O(\frac{1}{t})$ as $t\to\pm\infty$ with $a_j\neq 0$ and $\textup{Im}(b_j)\neq 0$.
This leads to the soliton resolution $u(t,x)=\sum_{j=1}^N\frac{C_j(t)}{x-\bar{a}_jt-\bar{b}_j}+\eps(t,x)$
in $H^s$ for all $s\geq 0$.
In the non-generic case, there is $j_0$ such that $E_{j_0}(t)=\textup{Re}(b_{j_0})+O(\frac{1}{t})$
as $t\to\pm\infty$. Then, $\textup{Im}(E_{j_0})=O(\frac{1}{t})$
and thus, one of the poles of the solution approaches the real line as $|t|\to\infty$.
This causes $\|u(t)\|_{H^s}$ to grow to $\infty$ if $s>1/2$.
\medskip

In Section 6, we prove Theorem \ref{thm:action-angle} and Corollary \ref{cor:toroidal cilinders}.
The Szeg\"o equation is an infinite dimensional, completely integrable system.
The Lax pair structure yields the existence of an
infinite sequence of prime integrals
$J_{2n}=(u,H_u^{2n-2}u)$, $n\in\N$.
Since the finite dimensional manifolds $\M(N)$ are invariant under the flow,
by restricting the Szeg\"o equation to $\M(N)$, we obtain a  $4N$-dimensional,
completely integrable system. The common level sets of the prime integrals $J_{2n}$ are not compact.
Then, a generalization of the Liouville-Arnold
theorem \cite{FS,FGS} to the case of a $4N$-dimensional, completely integrable system, with non-compact level sets, states
the {\it existence} of generalized action-angle coordinates, if certain conditions are satisfied.
In these coordinates ($2N$ invariant action coordinates, $k$ angle coordinates belonging to $\T$, and $2N-k$
generalized angle coordinates belonging to $\R$) the equation can be easily integrated.
 In Theorem \ref{thm:action-angle}, we {\it explicitly introduce}
generalized action-angle coordinates in terms of the spectral data.

Our strategy is to use the Szeg\"o hierarchy, i.e. the
the infinite family of completely integrable systems corresponding to the
Hamiltonian vector fields of $J_{2n}$.
The difficulty consists in
proving that $\gamma_j=\textup{Re}\,(Te_j,e_j)$ are the generalized angles.

\section{Explicit formula for the solution in the case of rational function initial data }

\bigskip

In this section we find the explicit formula for the solution in the case of
rational functions data.

\begin{lemma}\label{lemma:g=1-b}
Let $u=\frac{A}{B}\in\M(N)$, where $A$ and $B$ are relatively prime, $B(0)=1$,
$B(x)=(x-p_1)^{m_1}\dots(x-p_k)^{m_k}$, $m_1+\dots m_k=N$,
and $\textup{Im}(p_j)<0$ for all $j=1,2,\dots,k$. Then $\textup {Ker} (H_u)=b_uL^2_+$, where
\begin{equation*}
b_u=\prod_{j=1}^k\frac{(x-\bar{p}_j)^{m_j}}{(x-p_j)^{m_j}}
\end{equation*}
and
\begin{equation}\label{eqn:g=1-b}
g=1-b_u.
\end{equation}
\end{lemma}

\begin{proof}
Let $f\in \text{Ker}(H_u)=\big(\textup{Ran}(H_u)\big)^{\perp}$. Then by equation \eqref{eq:ranH_u} we have that
\begin{equation*}
\Big(f,\frac{1}{(x-p_j)^{l_j}}\Big)=0 \text{ for all } j=1,2,\dots,k \text{ and } l_j=1,2,\dots,m_j.
\end{equation*}

\noi
By the residue theorem we have that
\begin{align}\label{residue thm}
\ft{\frac{1}{(x-p_j)^{l_j}}}(\xi)=\int_{\R}\frac{e^{-ix\xi}}{(x-p_j)^{l_j}}=\frac{2\pi (-i)^{l_j}}{(l_j-1)!}\xi^{l_j-1}e^{-ip_j\xi}.
\end{align}
Using the Plancherel formula, we obtain that
\begin{equation*}
0=(\hat{f}(\xi),\xi^{l_j-1}e^{-ip_j\xi})=\int e^{i\bar{p}_j\xi}\hat{f}(\xi)\xi^{l_j-1}d\xi
\end{equation*}

\noi
and thus $(D^{l_j-1}f)(\bar{p}_j)=0$, for all $j=1,2,\dots,k$ and $l_j=1,2,\dots,m_j$.
Then, the classical property \cite[Corollary 3.7.4, p.38]{Nikolskii} stating that if $f\in L^2_+$ is such that $f(\bar{p})=0$, $\text{Im}(p)<0$,
then $f(x)=\frac{x-\bar{p}}{x-p}f'(x)$ with $f'\in L^2_+$, applied recurrently to $f$, $Df$,...,$D^{m_j-1}f$
yields the formula for $b_u$. Using this formula we obtain
\begin{equation}\label{eqn:Pi(u bar{b_u})=0}
\Pi(u\bar{b}_u)=\Pi\Big(\frac{A}{(x-\bar{p}_1)^{m_1}\dots(x-\bar{p}_k)^{m_k}}\Big)=0.
\end{equation}

\noi
Moreover, equation \eqref{eq:ranH_u} yields that $1-b_u\in\text{Ran}(H_u)$ and by \eqref{eqn: def g} we have that
\begin{equation*}
H_u(1-b_u)=\Pi(u-u\bar{b}_u)=u=H_u(g).
\end{equation*}

\noi
Since $H_u$ is one to one on its range, we conclude that $1-b_u=g$.
\end{proof}

\begin{lemma}\label{lemma:ft of u}
If $u\in\M(N)$ and if $g$ is such that $u=H_u(g)$, then
\begin{equation*}
\hat{u}(\lambda)=(u,e^{i\lambda x}g) \text{ for all } \lambda>0.
\end{equation*}

\end{lemma}

\begin{proof}
Denoting by $\mathcal{F}$ the Fourier transform, we have that
\begin{align*}
\hat{u}(\lambda)&=\int e^{-i\lambda x}udx=\int e^{-i\lambda x}u(1-\bar{g})dx+\int e^{-i\lambda x}u\bar{g}dx\\
&=\mathcal{F}(u(1-\bar{g}))(\lambda)+(u,e^{i\lambda x}g)=\cj{\mathcal{F}(\bar{u}(1-g))}(-\lambda)+(u,e^{i\lambda x}g).
\end{align*}

\noi
By \eqref{eqn:bar{u}(1-g)} we have that $\bar{u}(1-g)\in L^2_+$.
Thus, the first term is the Fourier transform at $-\ld<0$ of
a function in $L^2_{+}$, and hence it is zero.
\end{proof}

\begin{lemma}
If $u(t)$ is the solution of the Szeg\"o equation corresponding to the initial condition
$u_0\in\M(N)$ at time $t$ and $g(t)\in\textup{Ran}(H_{u(t)})$ is such that
$u(t)=H_{u(t)}(g(t))$, then we have:
\begin{equation}\label{eqn:Uastu}
U^{\ast}(t)u(t)=e^{-i\frac{t}{2}H^2_{u_{0}}}u_0,
\end{equation}

\begin{equation*}
U^{\ast}(t)g(t)=e^{i\frac{t}{2}H^2_{u_{0}}}g_0.
\end{equation*}
\end{lemma}

\begin{proof}
Differentiating with respect to $t$ and using equations \eqref{eqn:Bu},
\eqref{eq:szego}, and \eqref{eqn:H_u}, we have
\begin{align*}
\frac{d}{dt}U^{\ast}u=&-U^{\ast}B_uu-iU^{\ast}T_{|u|^2}u=
-U^{\ast}(-iT_{|u|^2}u+\frac{i}{2}H^2_uu)-iU^{\ast}T_{|u|^2}u\\
=&-\frac{i}{2}U^{\ast}H^2_u u=-\frac{i}{2}H^2_{u_{0}}U^{\ast}u.
\end{align*}

\noi
Since $U^{\ast}(0)=U(0)=I$, this yields the first equality. By equation \eqref{eqn:H_u}
and using the fact that the operator $H_u$ is skew-symmetric,
we can rewrite \eqref{eqn:Uastu} as
\[H_{u_0}\Big(U^{\ast}(t)g(t)-e^{i\frac{t}{2}H^2_{u_{0}}}g_0\Big)=0.\]

\noi
By \eqref{eqn:URan} we have that $U^{\ast}(t)g(t)-e^{i\frac{t}{2}H^2_{u_{0}}}g_0\in \text{Ran}(H_{u_0})$ and
since $H_{u_0}$ is one to one on $\text{Ran}(H_{u_0})$, the second equality follows.
\end{proof}

In the following we denote the unitary operator
 $e^{i\frac{t}{2}H^2_{u_{0}}}$ by $W(t)$. The skew-symmetry of the
Hankel operator $H_{u_0}$ yields
\begin{equation}\label{eqn:H_u_0W}
H_{u_0}W=W^{\ast}H_{u_0}.
\end{equation}

\noi
We also set
\begin{equation}\label{eqn: def e}
\tilde{e}(t):=U^{\ast}(t)g(t)=e^{i\frac{t}{2}H^2_{u_{0}}}g_0=W(t)g_0.
\end{equation}

With these notations we have, by equation \eqref{eqn:H_u}, that
\begin{equation}\label{eqn:u1}
u(t)=H_{u(t)}(g(t))=U(t)H_{u_0}U^{\ast}(t)g(t)=U(t)(H_{u_0}\tilde{e}(t)).
\end{equation}

\begin{definition}
Let us denote by $P_u$ the orthogonal projection on $\textup{Ran}(H_{u})$.
We also denote by $T_{\ld}$, $\ld>0$, the compressed shift operators acting on $\textup{Ran}(H_u)$ by
\begin{equation*}
T_{\ld}f=P_u(e^{i\ld x}f), \text{ for all } f\in \textup{Ran}(H_u).
\end{equation*}

\noi
If $u(t)$ is the solution of the Szeg\"o equation with initial condition $u_0$
and $T_{\ld}(t)$ acts on $\textup{Ran}(H_{u(t)})$, then
we define the operators $S_{\lambda}(t)$, $\lambda>0$, $t\in\R$
on $\textup{Ran}(H_{u_0})$ by
\begin{equation*}
S_{\lambda}(t)f=U^{\ast}(t)T_{\ld}(t)U(t)=P_{u_0}U^{\ast}(t)e^{i\lambda x}U(t)f.
\text{ for all } f\in \textup{Ran}(H_{u_0}),
\end{equation*}
\end{definition}

Notice that using \eqref{eqn:URan}, we have
\begin{equation}\label{eqn:e^{ild x}g}
P_{u(t)}e^{i\lambda x}g(t)=U(t)(P_{u_0}U^{\ast}(t)e^{i\lambda x}U(t))U^{\ast}(t)g(t)
=U(t)(S_{\lambda}(t)\tilde{e}).
\end{equation}

\begin{definition}
Let $u=\frac{A}{B}\in\M(N)$, where $A$ and $B$ are relatively prime,
$B(0)=1$, $B(x)=(x-p_1)^{m_1}\dots(x-p_k)^{m_k}$, $m_1+\dots m_k=N$,
and $\textup{Im}(p_j)<0$ for all $j=1,2,\dots,k$. For all $f\in\textup{Ran}(H_u)$,
\begin{equation*}
f=\sum_{j=1}^k \frac{\alpha_j}{x-p_j}
+\sum_{j=1}^k \sum_{l_j=2}^{m_j}\frac{\beta_j^l}{(x-p_j)^{l_j}},
\end{equation*}
\noi
we define
\begin{equation*}
\Lambda(f):=\sum_{j=1}^k \alpha_j=\lim_{x\to\infty}xf(x).
\end{equation*}
The infinitesimal shift operator is the linear operator $T$ defined on $\textup{Ran}(H_u)$ by:
\begin{equation}\label{eqn:defT}
T(f)=xf-\Lambda(f)b_u.
\end{equation}

\noi
Notice that by \eqref{eq:ranH_u}, we have that
$T(f)\in\textup{Ran}(H_u)$ for all $f\in\textup{Ran}(H_u)$.

If $u(t)$ is the solution of the Szeg\"o equation with initial condition $u_0$ and $T(t)$
is the operator $T$ acting on $\textup{Ran}(H_{u(t)})$, we introduce the family of
operators $S(t)$ acting on $\textup{Ran}(H_{u_0})$, by
\begin{equation*}
S(t)=U^{\ast}(t)T(t)U(t).
\end{equation*}

\end{definition}

\begin{lemma}\label{lemma:eigen T}
The eigenvalues of $T$ and $S$ are
the complex conjugates of the poles of $u$. In particular, the eigenvalues
of $T$ and $S$ have strictly positive imaginary part.
\end{lemma}

\begin{proof}
Since $T$ and $S$ are conjugated, they have the same eigenvalues.
If $Tf=\ld f$, then we have that $(x-\ld)f=\Ld(f)b_u$. Taking $x=\ld$, we obtain that $b_u(\ld)=0$.
Then, Lemma \ref{lemma:g=1-b} yields that $\ld=\bar{p}_j$.
\end{proof}

\begin{remark}
\rm
Notice that we can extend the definition of $\Ld$ to
\begin{equation*}
T_{|u|^2}\big(\text{Ran}(H_u)\big)
=\Big\{\sum_{j=1}^k \sum_{l_j=1}^{2m_j}\frac{\beta_j^l}{(x-p_j)^{l_j}};\,\, \beta_j^l\in\C  \Big\}.
\end{equation*}

\noi
We then use formula \eqref{eqn:defT} to extend the definition of $T$ to $T_{|u|^2}\big(\text{Ran}(H_u)\big)$.
\end{remark}

\begin{lemma}\label{lemma:S in terms of Sld}
The operator $iS$ is the infinitesimal generator of the semi-group  $S_{\ld}$, i.e. $S_{\ld}=e^{i\ld S}$ for all $\ld>0$.
\end{lemma}

\begin{proof}
Because of the definitions of $S$ and $S_{\ld}$ in terms of
$T$ and $T_{\ld}$, it is enough to prove that
\begin{align*}
-i\frac{d}{d\ld}_{|\ld =0}T_{\ld}f=TT_{\ld|\ld =0}f,
\end{align*}

\noi
where $T$ and $T_{\ld}$ act on $\text{Ran}(H_u)$.

Define the linear operator $L:\textup{Hol}(\C_+)\to \C^N$ by
\begin{align*}
L:f \mapsto \big\{\partial_x^m f(\bar{p}_j)\big| j\in\{1,2,\dots,k\}, m\in\{0,2,\dots,m_j-1\}\big\},
\end{align*}

\noi
where $p_1,p_2,\dots,p_k$ are the poles of $u$ and $m_j$ is the multiplicity of
the pole $p_j$. Then, we have that
$\textup{Ker} L=b_u\textup{Hol}(\C_+)$, where $b_u=\prod_{j=1}^k\Big(\frac{x-\bar{p}_j}{x-p_j}\Big)^{m_j}$.
In particular, $L_{|\textup{Ran}(H_u)}:\textup{Ran}(H_u)\to L(\textup{Ran}(H_u))$ is a isomorphism.
Since $T_{\ld}f,Tf\in\textup{Ran}(H_u)$ for all $f\in\textup{Ran}(H_u)$,
the lemma is proved once we show that $L(-i\frac{d}{d\ld}_{|\ld =0}T_{\ld}f)=L(Tf)$.
This is indeed true since $L(b_u)=0$, $L(h)=L(P_uh)$ for all $h\in L^2_+$, and
\begin{align*}
\frac{d}{d\ld}_{|\ld=0}L(T_{\ld}f)&=\frac{d}{d\ld}_{|\ld=0}L(P_u(e^{i\ld x}f))
=\frac{d}{d\ld}_{|\ld=0}L(e^{i\ld x}f)\\
&=iL(xf)=iL(xf-\Ld(f)b_u)=iL(Tf).
\end{align*}
\end{proof}

\begin{proposition}\label{lemma:formula for the solution}
If $u(t)$ is the solution of the Szeg\"o equation
corresponding to the initial data $u_0\in\M(N)$,
 then the following formula holds:
\begin{equation}\label{eqn:formula for u}
u(t,x)=\frac{i}{2\pi}\Big(u_0,W(t)(S-xI)^{-1}W(t)g_0\Big).
\end{equation}
\end{proposition}

\begin{proof}
Using the Cauchy integral formula, Plancherel's identity, equation \eqref{residue thm}, Lemma \ref{lemma:ft of u},
equations \eqref{eqn:u1} and \eqref{eqn:e^{ild x}g}, the fact that
$U(t)$ are unitary operators,
equation \eqref{eqn:H_u_0W}, and Lemma \ref{lemma:S in terms of Sld}, we have for $\textup{Im}z>0$ that
\begin{align*}
u(z,t)&=\frac{1}{2\pi i}\int_0^{\infty}\frac{u(x)}{x-z}dx
=\frac{1}{4\pi ^2 i}\int_0^{\infty}\ft{u}(t,\lambda)\cj{\ft{\frac{1}{x-\bar{z}}}}d\lambda
=\frac{1}{2\pi}\int_0^{\infty}e^{iz\ld}\ft{u}(t,\lambda)d\lambda\\
&=\frac{1}{2\pi}\int_0^{\infty}e^{iz\ld}(u(t),e^{i\ld x}g(t))d\ld
=\frac{1}{2\pi}\int_0^{\infty}e^{iz\ld}(u(t),P_{u(t)}e^{i\ld x}g(t))d\ld\\
&=\frac{1}{2\pi}\int_0^{\infty}e^{iz\ld}\Big(U(t)(H_{u_0}\tilde{e}),U(t)(S_{\ld}(t)\tilde{e})\Big)d\ld
=\frac{1}{2\pi}\int_0^{\infty}e^{iz\ld}\Big(H_{u_0}\tilde{e},S_{\ld}(t)\tilde{e}\Big)d\ld\\
&=\frac{1}{2\pi}\int_0^{\infty}e^{iz\ld}\Big(H_{u_0}(W(t)g_0),S_{\ld}(t)W(t)g_0\Big)d\ld
=\frac{1}{2\pi}\int_0^{\infty}e^{iz\ld}\Big(W(t)^{\ast}H_{u_0}g_0,S_{\ld}(t)W(t)g_0\Big)d\ld\\
&=\frac{1}{2\pi}\Big(W(t)^{\ast}u_0,\int_0^{\infty}e^{\lambda(iS-i \bar{z}I)}d\ld W(t)g_0\Big)
=\frac{1}{2\pi}\Big(u_0,W(t)(iS-i \bar{z}I)^{-1}W(t)g_0\Big).
\end{align*}

The above formula also holds for $x\in\R$ since, by Lemma \ref{lemma:eigen T}, the eigenvalues of $S$
are not real numbers.
\end{proof}

Notice that in this formula for $u(t)$, the operator $S(t)$
is defined using $U(t)$
whose definition depends on $u(t)$. Our goal is to
characterize $S(t)$ without using $U(t)$.
In order to do that, we need to determine the derivative in time of $S(t)h$,
for any $h\in\textup{Ran}(H_{u_0})$.
This derivative is expressed in terms of commutators of $T$ with
Hankel and Toeplitz operators, that we compute in the following.

The below formula, that can be proved by passing into the Fourier space, will be useful:

\begin{equation}\label{eqn:basic xf}
\Pi(xf)=x\Pi(f)+\frac{1}{2\pi i}\int f,
\end{equation}

\noi
if $f\in L^1(\R)\cap L^2(\R)$ and $xf\in L^2(\R)$.

\begin{lemma}\label{lemma: Ld(H_uf)}
If $u\in\M(N)$ and $f\in\textup{Ran}(H_u)$, then
\begin{align}
\Lambda(H_uf)&=-\frac{1}{2\pi i}\int u\bar{f},\label{eqn:LdH} \\
\Lambda (f)&=-\frac{1}{2\pi i}(f,g) \text{ for all } f\in\textup{Ran}(H_u)\label{Ld(f)},\\
\Lambda(T_{|u|^2}f)&=-\frac{1}{2\pi i}\int |u|^2f\label{eqn:LdT}.
\end{align}

\end{lemma}

\begin{proof}
The result follows once we prove that for all $f_1,f_2\in\M(N)$
that have the same poles, $p_1,\dots,p_k$, we have

\begin{equation}\label{eqn:Ld(Pi(f1f2))}
\Ld(\Pi(f_1\bar{f}_2))=-\frac{1}{2\pi i}\int f_1\bar{f}_2.
\end{equation}

\noi
Indeed, \eqref{eqn:LdH} follows taking $f_1=u$, $f_2=f$ and \eqref{eqn:LdT}
follows taking $f_1=uf$, $f_2=u$. Then, \eqref{Ld(f)} is a direct consequence of \eqref{eqn:LdH}.
In order to prove \eqref{eqn:Ld(Pi(f1f2))} we decompose $f_1\bar{f}_2$
into simple rational fractions:
\begin{equation*}
f_1\bar{f}_2=\frac{A_1}{x-p_1}+\dots+\frac{A_k}{x-p_k}+\frac{B_1}{x-\bar{p}_1}
+\dots+\frac{B_k}{x-\bar{p}_k}+\sum_{j=1}^k\sum_{l_j=2}^{m_j}\frac{A_j^{l_j}}{(x-p_j)^{l_j}}
+\sum_{j=1}^k\sum_{l_j=2}^{m_j}\frac{B_j^{l_j}}{(x-\bar{p}_j)^{l_j}}.
\end{equation*}

\noi
Since $\text{Im}(p_j)<0$ for all $j=1,2,\dots,k$, the residue theorem yield:
\begin{equation*}
\int f_1\bar{f}_2=-2\pi i(A_1+\dots+A_k)=-2\pi i\Ld(\Pi(f_1\bar{f}_2)).
\end{equation*}
\end{proof}

\begin{lemma}\label{lemma:comutators}
For all $h\in\textup{Ran}(H_u)$ we have
\begin{align}
[T,T_{|u|^2}]h=&-\frac{1}{2\pi i}\Big(\int |u|^2h\Big) g
+\Ld (h)|u|^2b_u,\label{eqn:comutator T}\\
[T,H_u^2]h=&- \frac{1}{2\pi i}\Big(\int |u|^2h\Big)g
+\frac{1}{2\pi i}\Big(\int \bar{u}h\Big) u.\label{eqn:comutator H^2}
\end{align}

\end{lemma}

\begin{proof}
Using equations \eqref{eqn:basic xf}, \eqref{eqn:LdT}, and \eqref{eqn:Pi(u bar{b_u})=0}, we have
\begin{align*}
[T,T_{|u|^2}]h=&xT_{|u|^2}h-\Ld(T_{|u|^2}h)b_u-T_{|u|^2}(xh-\Ld(h)b_u)\\
=&\Big(x\Pi(|u|^2h)-\Pi(x|u|^2h)\Big)-\Ld(T_{|u|^2}h)b_u+\Ld(h)\Pi(|u|^2b_u)\\
=&- \frac{1}{2\pi i}\Big(\int |u|^2h\Big)
+\frac{1}{2\pi i}\Big(\int |u|^2f\Big)b_u+\Ld(h) |u|^2b_u.
\end{align*}

\noi
The first formula now follows using equation \eqref{eqn:g=1-b}.
Secondly, using equations \eqref{eqn:Pi(u bar{b_u})=0}, \eqref{eqn:basic xf} twice,
\eqref{eqn:LdH}, and \eqref{eqn:g=1-b}, we have
\begin{align*}
[T,H_u^2]h=&xH_u^2h-\Ld(H_u^2h)b_u-H_u\Big(\Pi(xu\bar{h})-\Ld(h)\Pi(u\bar{b}_u)\Big)\\
=&xH_u^2h-H_u\big(\Pi(xu\bar{h}))-\Ld(H_u^2h)b_u\\
=&\Pi(xu\cj{H_uh})-\frac{1}{2\pi i}\int u\cj{H_uh}-H_u\Big(x\Pi(u\bar{h})
+\frac{1}{2\pi i}\int u\bar{h}\Big)-\Ld(H_u^2h)b_u\\
=&\Pi(xu\cj{H_uh})-\frac{1}{2\pi i}\int u\cj{H_uh}-\Pi(ux\cj{H_uh})
+\frac{1}{2\pi i}\Big(\int \bar{u}h\Big) u-\Ld(H_u^2h)b_u\\
=&-\frac{1}{2\pi i}\Big(\int u\cj{H_uh}\Big)(1-b_u)
+\frac{1}{2\pi i}\Big(\int \bar{u}h\Big) u\\
=&-\frac{1}{2\pi i}(u,u\bar{h})g
+\frac{1}{2\pi i}\Big(\int \bar{u}h\Big) u.
\end{align*}
\end{proof}

\begin{lemma}\label{lemma:d_tS}
For all $h\in\textup{Ran}(H_{u_0})$, we have
\begin{equation}\label{eqn:d_tS}
P_{u_0}\frac{d}{dt}S(t)h=\frac{1}{4\pi}\Big((h,H^2_{u_0}\tilde{e})\tilde{e}
+(h,H_{u_0}\tilde{e})H_{u_0}\tilde{e}\Big).
\end{equation}

\end{lemma}

\begin{proof}
Using equation \eqref{eqn:d_tB_u}, we have that
\begin{equation*}
P_{u_0}i\frac{d}{dt}S(t)h=P_{u_0}U^{\ast}[T,T_{|u|^2}-\frac{1}{2}H_u^2]Uh
+P_{u_0}U^{\ast}(i\frac{d}{dt}T(t))Uh.
\end{equation*}

\noi
Then, by Lemma \ref{lemma:comutators}, equation \eqref{eqn:g=1-b},
$b_u=1-g$, equations \eqref{eqn: def e} and \eqref{eqn:u1}, we have
\begin{align*}
P_{u_0}i\frac{d}{dt}S(t)h=&P_{u_0}U^{\ast}\bigg(-\frac{1}{2\pi i}\big(\int |u|^2Uh\big) g+\Ld (Uh)|u|^2b_u\\
&+\frac{1}{4\pi i}\big(\int |u|^2Uh\big) g
-\frac{1}{4\pi i}\big(\int \bar{u}Uh\big) u\bigg)+P_{u_0}U^{\ast}\big(ig'\Ld(Uh)\big)\\
=&-\frac{1}{4\pi i}\big(\int |u|^2Uh\big) \tilde{e}-\frac{1}{4\pi i}\big(\int \bar{u}Uh\big) H_{u_0}\tilde{e}\\
&+\Ld (Uh)P_{u_0}U^{\ast}(|u|^2b_u)+\Ld(Uh)P_{u_0}U^{\ast}(ig').
\end{align*}

\noi
In order to compute $g'(t)$, we will differentiate the equality $u=H_{u}g$. We obtain:
\begin{align*}
-iT_{|u|^2}u=[B_u,H_u]g+H_u(g')=-iT_{|u|^2}H_ug-iH_{u}T_{|u|^2}g+iH_u^3g+H_{u}(g').
\end{align*}

\noi
Then, $H_{u}(g'+i\Pi(|u|^2(g-1)))=0$ and thus by \eqref{eqn:Pi(u bar{b_u})=0} we have $P_u(ig')=-P_u\Pi(|u|^2b_u)=-P_u(|u|^2b_u)$.
Consequently, by \eqref{eqn:URan} we have
\begin{equation*}
P_{u_0}U^{\ast}(ig')=U^{\ast}P_u(ig')=-U^{\ast}P_u(|u|^2b_u)=-P_{u_0}U^{\ast}(|u|^2b_u).
\end{equation*}

\noi
Therefore we obtain
\begin{align*}
P_{u_0}\frac{d}{dt}S(t)h=\frac{1}{4\pi }\big(\int |u|^2Uh\big) \tilde{e}
+\frac{1}{4\pi }\big(\int \bar{u}Uh\big) H_{u_0}\tilde{e}.
\end{align*}

To conclude, we only need to rewrite the two parenthesis
so that they do not depend on $U$. By equation \eqref{eqn:H_u},
the definitions of $g$, $\tilde{e}$, and  equation \eqref{sym H_u}, we have:
\begin{align*}
\int |u|^2Uh&=(u,u\cj{Uh})=(u,\Pi(u\cj{Uh}))=(u,H_u(Uh))=(u,UH_{u_0}h)=(U^{\ast}H_ug,H_{u_0}h)\\
&=(H_{u_0}U^{\ast}g,H_{u_0}h)=(H_{u_0}\tilde{e},H_{u_0}h)=(h,H_{u_0}^2\tilde{e}).
\end{align*}

\begin{align}\label{eqn:int bar{u}Uh}
\int \bar{u}Uh=(Uh,u)=(Uh,H_ug)=(h,U^{\ast}H_ug)=(h,H_{u_0}U^{\ast}g)=(h,H_{u_0}\tilde{e}).
\end{align}
\end{proof}

In order to express $S$ without using $U(t)$, we also need to determine the adjoint $S^{\ast}$
of the operator $S$
and prove the commutation relation $S^{\ast}H_{u_0}=H_{u_0}S$. We first determine $T^{\ast}$.

\begin{lemma}\label{lemma Tast}
The adjoint of the operator $T$ on $\textup{Ran} (H_u)$ is the operator $T^{\ast}$ defined by
\begin{equation*}
T^{\ast}f=xf-\Ld (f), \text{ for all } f\in\textup{Ran}(H_u).
\end{equation*}

\end{lemma}

\begin{proof}
By equations \eqref{eqn:basic xf}, \eqref{eqn:g=1-b}, and \eqref{eqn:LdH}, for all $f_1,f_2\in L^2_+$ we have that:
\begin{align*}
(TH_uf_1,H_uf_2)&=(xH_uf_1-\Ld (H_uf_1)b_u,u\bar{f}_2)
=(H_uf_1,xu\bar{f}_2)-\Ld (H_uf_1)(b_uf_2,u)\\
&=(H_uf_1,\Pi(xu\bar{f}_2))-\Ld (H_uf_1)(f_2(1-g),u)\\
&=\Big(H_uf_1,x\Pi(u\bar{f}_2)+\frac{1}{2\pi i}\int u\bar{f}_2\Big)-\Ld (H_uf_1)\Big((f_2,u)-(f_2,\bar{g}u)\Big)\\
&=(H_uf_1,xH_uf_2-\Ld (H_uf_2))-\Ld (H_uf_1)\Big((f_2,u)-(f_2,H_ug)\Big)\\
&=(H_uf_1,xH_uf_2-\Ld (H_uf_2)).
\end{align*}

\noi
Hence $T^{\ast}\big(H_uf_2\big)=xH_uf_2-\Ld (H_uf_2)$, for all $f_2\in L^2_+$.
\end{proof}

\begin{lemma}\label{lemma:relation S,Sast}
\begin{equation*}
S^{\ast}H_{u_0}=H_{u_0}S
\end{equation*}

\noi
and
\begin{equation}\label{eq:S-Sast}
S=S^{\ast}-\frac{1}{2\pi i}(\cdot,\tilde{e})\tilde{e}.
\end{equation}

\end{lemma}

\begin{proof}
By projecting equation \eqref{tilde T_ld} on $\textup{Ran}(H_u)$, we obtain
$T_{\ld}^{\ast}H_u=H_uT_{\ld}$. Then, by Lemma \ref{lemma:S in terms of Sld}
it follows that
$T^{\ast}H_u=H_uT$.
This and equation \eqref{eqn:H_u} yield for all $h\in\text{Ran}(H_{u_0})$ that
\begin{align*}
H_{u_0}Sh=H_{u_0}U^{\ast}TUh=U^{\ast}H_uTUh=U^{\ast}T^{\ast}H_uUh
=U^{\ast}T^{\ast}UH_{u_0}h=S^{\ast}H_{u_0}h.
\end{align*}

\noi
Notice that \eqref{Ld(f)} yields that
\begin{align}\label{T-Tast}
T=T^{\ast}-\frac{1}{2\pi i}(\cdot,g)g.
\end{align}

\noi
Then, \eqref{eq:S-Sast} follows immediately by conjugating the above relation with $U(t)$.
\end{proof}

\begin{proof}[Proof of Theorem \ref{thm:general formula}]
By conjugating equation \eqref{eqn:comutator H^2} by $U(t)$, we obtain:
\begin{align*}
[H_{u_0}^2,S]h=\frac{1}{2\pi i}\Big((h,H_{u_0}^2\tilde{e})\tilde{e}-(h,H_{u_0}\tilde{e})H_{u_0}\tilde{e}\Big),
\end{align*}

\noi
for all $h\in\text{Ran}(H_{u_0})$. Applying this to $h=e_j$ we have
\begin{align*}
(H_{u_0}^2-\ld_j^2)Se_j=\frac{\ld_j}{2\pi i}\Big(\ld_j(e_j,\tilde{e})\tilde{e}
-(\tilde{e},e_j)H_{u_0}\tilde{e}\Big).
\end{align*}

\noi
Suppose that $\ld_j$ is an eigenvalue of multiplicity $m_j$ and that
$M_j$ is the set of all indices $k$ such that $H_{u_0}e_{k}=\ld_j e_{k}$.
Plugging $\tilde{e}=\sum_{k=1}^N(\tilde{e},e_k)e_k$ in the above formula we have:
\begin{align*}
(H_{u_0}^2-\ld_j^2)Se_j
&=\frac{\ld_j}{2\pi i}\sum_{k\notin M_j}\Big(\ld_j (e_j,\tilde{e})(\tilde{e},e_k)-\ld_k (\tilde{e},e_j)(e_k,\tilde{e})\Big)e_k
\\&+\frac{\ld_j^2}{2\pi i}\sum_{k\in M_j}\Big((e_j,\tilde{e})(\tilde{e},e_k)-(\tilde{e},e_j)(e_k,\tilde{e})\Big)e_k.
\end{align*}

\noi
Since
\begin{align*}
(\tilde{e},e_j)=(e^{i\frac{t}{2}H^2_{u_0}}g_0,e_j)=e^{i\frac{t}{2}\ld_j^2}(g_0,e_j)=e^{i\frac{t}{2}\ld_j^2}\beta_j,
\end{align*}

\noi
we obtain
\begin{align*}
(H_{u_0}^2-\ld_j^2)Se_j
&=\frac{\ld_j}{2\pi i}\sum_{k\notin M_j}\Big(\ld_j e^{i\frac{t}{2}(\ld_k^2-\ld_j^2)}\cj{\beta}_j\beta_k
-\ld_k e^{i\frac{t}{2}(\ld_j^2-\ld_k^2)}\beta_j\cj{\beta}_k\Big)e_k
\\&+\frac{\ld_j^2}{2\pi i}\sum_{k\in M_j}(\cj{\beta}_j\beta_k-\beta_j\cj{\beta}_k)e_k
\end{align*}

\noi
Writing
\begin{align*}
S(t)e_j=\sum_{k=1}^Nc_j^k(t)e_k,
\end{align*}

\noi
we have
\begin{align*}
(H_{u_0}^2-\ld_j^2) S(t)e_j=\sum_{k\notin M_j}(\ld_k^2-\ld_j^2)c_j^k(t)e_k.
\end{align*}

\noi
Identifying the coefficients of $(H_{u_0}^2-\ld_j^2) S(t)e_j$ in the basis $\{e_k\}_{k=1}^N$, we obtain that
\begin{align}\label{eqn:cj{beta_j}beta_i in R}
\cj{\beta}_j\beta_{k}\in\R, \text{ for all } k\in M_j
\end{align}

\noi
and
\begin{align*}
c_j^k(t)=\frac{\ld_j}{2\pi i(\ld_k^2-\ld_j^2)}\Big(\ld_j e^{i\frac{t}{2}(\ld_k^2-\ld_j^2)}\cj{\beta}_j\beta_k
-\ld_k e^{i\frac{t}{2}(\ld_j^2-\ld_k^2)}\beta_j\cj{\beta}_k\Big)
\end{align*}

\noi
for all $k\notin M_j$.
Finally, we determine $c_j^k(t)$ for $k\in M_j$ using Lemma \ref{lemma:d_tS}:
\begin{align*}
\frac{d}{dt}c_j^k(t)&=\frac{d}{dt}(S(t)e_j,e_k)=(P_{u_0}\frac{d}{dt}S(t)e_j,e_k)
=\frac{\ld_j^2}{4\pi}\Big((e_j,\tilde{e})(\tilde{e},e_k)+(\tilde{e},e_j)(e_k,\tilde{e})\Big)\\
&=\frac{\ld_j^2}{4\pi}(\cj{\beta}_j\beta_k+\beta_j\cj{\beta}_k)=\frac{\ld_j^2}{2\pi}\cj{\beta}_j\beta_k.
\end{align*}

\noi
Therefore, for $k\in M_j$ we have
\begin{align*}
c_j^k(t)=\frac{\ld_j^2}{2\pi}\cj{\beta}_j\beta_kt+c_j^k(0),
\end{align*}

\noi
where $c_j^k(0)=(S(0)e_j,e_k)=(Te_j,e_k)$.
\end{proof}

\bigskip

\section{Extension of the formula to general initial data}

\begin{proof}[Proof of Corollary \ref{cor:gen case}]
The proof of Theorem \ref{thm:general formula} can be adapted
to the case of a general initial data, as long as
$u_0\in\textup{Ran}(H_{u_0})$, i.e. there exists $g_0\in\textup{Ran}(H_{u_0})$
such that $u_0=H_{u_0}(g_0)$. Writing $g_0=\sum_{j=1}^{\infty}(g_0,e_j)e_j$
in the basis $\{e_j\}_{j=1}^{\infty}$, the fact that
$g_0\in L^2(\R)$ is equivalent to $\sum_{j=1}^{\infty}|(g_0,e_j)|^2<\infty$.
Since $u_0=H_{u_0}(g_0)$ yields
$(u_0,e_j)=\ld_j(e_j,g)$ for all $j\in\N^{\ast}$, it follows that $\{\beta_j\}_{j=1}^{\infty}
=\{\frac{1}{\ld_j}(u_0,e_j)\}_{j=1}^{\infty}\in \ell^2$.

The main difference with the case of rational functions data is that
 $S$ is no longer a matrix,
but an operator acting between infinite dimensional spaces.
Then, the infinitesimal generator of the semi-group
$S_{\ld}$ is not $iS$, but its closure $i\bar{S}$
(like in Proposition \ref{prop:infinitesimal generator}).
This explains the operator $\bar{S}$ appearing in the explicit formula.
\end{proof}

\begin{proposition}\label{prop:xu}
Let $s\geq 1$. If $u_0\in H^s_+$ and $xu_0\in L^{\infty}(\R)$, then the corresponding solution of the
Szeg\"o equation satisfies $xu(t,x)\in L^{\infty}(\R)$ for all $t\in\R$.
\end{proposition}

\begin{proof}
The local well-posedness follows using a fixed point argument in the space $(L^{\infty}_t,X)$, where
\begin{align*}
X:=H^s_+(\R)\cap \big\{f\big|xf(x)\in L^{\infty}(\R)\big\}.
\end{align*}

By equation \eqref{eqn:basic xf}, the H\"older inequality, and Sobolev embedding, we have:
\begin{align*}
\Big\|x\int_0^T\Pi\big(|u(t)|^2u(t)\big)dt\Big\|_{L^{\infty}_{t,x}}&\leq T\big\|x\Pi\big(|u(t)|^2u(t)\big)\big\|_{L^{\infty}_{t,x}}\\
&=T\Big\|\Pi\big(x|u(t)|^2u(t)\big)-\frac{1}{2\pi i}\int |u(t)|^2u(t)dx \Big\|_{L^{\infty}_{t,x}}\\
&\leq T\Big\|\Pi\big(x|u(t)|^2u(t)\big)\Big\|_{L^{\infty}_{t,x}}+\frac{T}{2\pi }\Big\|\int |u(t)|^2u(t)dx \Big\|_{L^{\infty}_{t}}\\
&\leq T\big\|\Pi\big(x|u(t)|^2u(t)\big)\big\|_{L^{\infty}_tH^1_x}+\frac{T}{2\pi}\Big\|\int |u(t)|^2u(t)dx\Big\|_{L^{\infty}_{t}}\\
&\leq T\big\|x|u(t)|^2u(t)\big\|_{L^{\infty}_tH^1_x}+\frac{T}{2\pi}\|u\|_{L^{\infty}_{t,x}}\|u\|^2_{L^{\infty}_{t}L^2_x}\\
&\leq  T\big(4\|xu\|_{L^{\infty}_{t,x}}+\|u\|_{L^{\infty}_{t}H^s_x}\big)\|u\|^2_{L^{\infty}_tH^s_x}+\frac{T}{2\pi}\|u\|^3_{L^{\infty}_{t}H^s_x}.
\end{align*}

\noi
The global well-posedness is a consequence of the Brezis-G\"alouet estimate
\begin{align*}
\|u\|_{L^{\infty}(\R)}\leq C\|u\|_{H^{1/2}(\R)}\Bigg(\log \Big(2+\frac{\|u\|_{H^s}}{\|u\|_{H^{1/2}}}\Big)\Bigg)^{\frac{1}{2}},
\end{align*}

\noi
and of Gronwall's inequality.
\end{proof}

\begin{lemma}\label{lim u}
For all $u\in H^{1/2}_+$, we have that $u\in\cj{\textup{Ran}(H_u)}$.

Moreover, if $u\in H^s(\R)$, $s>\frac{1}{2}$ and $xu(x)\in L^{\infty}(\R)$, we have that $u=\lim_{\eps\to 0}H_u(\frac{1}{1-i\eps x})$.
\end{lemma}

\begin{proof}
For $h\in L^2_+$, we have that
\begin{align}\label{eqn:eps}
(u,h)=\lim_{\eps\to 0}\Big(u,\frac{h}{1-i\eps x}\Big)
=\lim_{\eps\to 0}\Big(u\bar{h},\frac{1}{1-i\eps x}\Big)
=\lim_{\eps\to 0}\Big(H_uh,\frac{1}{1-i\eps x}\Big).
\end{align}

\noi
Taking $h\in\textup{Ker}(H_u)$, it follows that $(u,h)=0$ and $u\in(\textup{Ker}(H_u))^{\perp}=\cj{\textup{Ran}(H_u)}$.

By \eqref{sym H_u}, the above equation also yields that for all $h\in L^2_+$, we have that
\begin{align*}
(u,h)=\lim_{\eps\to 0}\Big(H_u\big(\frac{1}{1-i\eps x}\big),h\Big).
\end{align*}

\noi
Then, $H_u\big(\frac{1}{1-i\eps x}\big)$ converges weakly to $u$ in $L^2_+$.
We now intend to prove that, if $u\in H^s(\R)$ and $xu(x)\in L^{\infty}(\R)$, then $\big\|H_u\big(\frac{1}{1-i\eps x}\big)\big\|_{L^2}\to\|u\|_{L^2}$.
This yields that the convergence is strong in $L^2_+$.

Computing the Fourier transform with the residue theorem, we have that
\begin{align}\label{eqn: H_u(1/1+ieps x)}
H_u\big(\frac{1}{1-i\eps x}\big)=\Pi\Big(\frac{u(x)}{1+i\eps x}\Big)=
\frac{1}{i\eps}\Pi\Big(\frac{u(x)}{x-\frac{i}{\eps}}\Big)=\frac{1}{i\eps}\cdot\frac{u(x)-u(\frac{i}{\eps})}{x-\frac{i}{\eps}}
=\frac{u(x)-u(\frac{i}{\eps})}{1+i\eps x}.
\end{align}

\noi
By the Sobolev embedding $H^s(\R)\subset L^{\infty}(\R)$ for $s>\frac{1}{2}$,
we have that there exists $C_0>0$ such that $|u(x)|\leq C_0$ for all $x\in\R$. Since $u$ is a holomorphic function in $\C_+$,
we can write using the Poisson integral
\begin{align*}
u(z)=\frac{1}{\pi}\int_{-\infty}^{\infty}\textup{Im}z\frac{u(x)}{|z-x|^2}dx,
\end{align*}

\noi
for all $z\in \C_+$. Then
\begin{align*}
|u(z)|\leq\frac{C_0}{\pi}\textup{Im}z\int_{-\infty}^{\infty}\frac{1}{|z-x|^2}dx=C_0,
\end{align*}

\noi
for all $z\in\C_+$. Thus, $u$ is bounded in $\C_+\cup\R$. Similarly, since $xu(x)\in L^{\infty}(\R)$,
we have that $zu(z)$ is bounded in $\C_+\cup\R$ by a constant $C_1$. In particular, we have that
$\frac{i}{\eps}u(\frac{i}{\eps})\leq C_1$ and thus
$\lim_{\eps\to 0}u(\frac{i}{\eps})=0$. Then, by
\eqref{eqn: H_u(1/1+ieps x)}, we have that $H_u\big(\frac{1}{1-i\eps x}\big)$
converges pointwise to $u(x)$. Furthermore,
\begin{align*}
\Big|H_u\big(\frac{1}{1-i\eps x}\big)\Big|^2&=\Big|\frac{u(x)-u(\frac{i}{\eps})}{1+i\eps x}\Big|^2
\leq |u(x)-u(\frac{i}{\eps})|^2 \leq 2(|u(x)|^2+|u(\frac{i}{\eps})|^2)\leq 4C_0
\end{align*}

\noi
and
\begin{align*}
\Big|H_u\big(\frac{1}{1-i\eps x}\big)\Big|^2&
\leq\frac{2(|u(x)|^2+|u(\frac{i}{\eps})|^2)}{1+\eps^2x^2}\leq \frac{\frac{C_1}{x^2}+C_1\eps^2}{1+\eps^2x^2}=\frac{C_1}{x^2}.
\end{align*}

\noi
Then, the functions $\Big|H_u\big(\frac{1}{1-i\eps x}\big)\Big|^2$ are bounded by
an integrable function. By the dominated convergence theorem, it follows that
$\big\|H_u\big(\frac{1}{1-i\eps x}\big)\big\|_{L^2}\to\|u\|_{L^2}$. Hence,
$H_u\big(\frac{1}{1-i\eps x}\big)\to u$ in $L^2_+$.
\end{proof}

The key point in extending the explicit formula for the solution to the case of general initial data
is the below definition of the operator $T^{\ast}: \textup{Ran}(H_u)\to L^2_+$,
\begin{align} \label{eq:def T star}
T^{\ast}(H_uf)=xH_u(f)+\frac{1}{2\pi i}(u,f).
\end{align}

\noi
If $xu\in L^{\infty}(\R)$, by \eqref{eqn:basic xf} we have that
\begin{align*}
T^{\ast}(H_uf)=\Pi (xu\bar{f}).
\end{align*}

\begin{remark}
If $u\in H^s_+$ for $s>\frac{1}{2}$ and $xu\in L^{\infty}(\R)$,
then the operator $T^{\ast}$ takes values in $\cj{\textup{Ran}(H_u)}$.
\end{remark}

\begin{proof}
For all $f\in \textup{Ran}(H_u)$ and $h\in\textup{Ker}(H_u)$, we have that
\begin{align*}
(T^{\ast}f,h)&=(\Pi(xu\bar{f}),h)=(xu\bar{f},h)=\lim_{\eps\to 0}(xu\bar{f}, \frac{h}{1-i\eps x})=
\lim_{\eps\to 0}(u\bar{h}, x\frac{f}{1-i\eps x})\\
&=\lim_{\eps\to 0}(H_uh, x\frac{f}{1-i\eps x})=0.
\end{align*}

\noi
Then, $T^{\ast}f\in (\textup{Ker}(H_u))^{\perp}=\cj{\textup{Ran}(H_u)}$.
\end{proof}

For $\ld>0$, we introduce the operators $T_{\ld}^{\ast}:L^2_+\to L^2_+$ by
\begin{align*}
T_{\ld}^{\ast}h(x)&=P_ue^{-i\ld x}\mathcal{F}^{-1}(\hat{h}(\xi)\chi_{[\ld,+\infty)}(\xi))= P_u\Big(e^{-i\ld x}h(x)-\frac{e^{-i\ld x}}{2\pi}\int_{\R} h(x-y)\frac{e^{i\ld y}-1}{iy}dy\Big).
\end{align*}

\noi
Then
\begin{align*}
\lim_{\ld\to 0}\frac{T_{\ld}^{\ast}h(x)-h(x)}{\ld}&=P_u\Big(-ixh(x)-\frac{1}{2\pi}\int_{\R}h(x)dx\Big).
\end{align*}

\noi
Let us now conjugate $T^{\ast}$ and $T_{\ld}^{\ast}$ using the operators $U(t)$.
We obtain $S^{\ast}(t)$ and $S^{\ast}_{\ld}(t)$:
\begin{align*}
S^{\ast}(t)=U^{\ast}(t)T^{\ast}U(t),\,\,\,\,\,\,\,\,\,\,\,\,\,S_{\ld}^{\ast}(t)=U^{\ast}(t)T_{\ld}^{\ast}U(t).
\end{align*}

\begin{proposition}\label{prop:infinitesimal generator}
The closure of the operator $-iS^{\ast}$ is the infinitesimal generator of the semi-group $S_{\ld}^{\ast}$.
Moreover, $\textup{Ran}(H_{u_0})$ is a core for the infinitesimal generator of the semi-group $S_{\ld}^{\ast}$.
\end{proposition}

\begin{proof}
If $h=H_uf\in \textup{Ran}(H_u)$, then we have
\begin{align*}
\lim_{\ld\to 0}\frac{T_{\ld}^{\ast}h(x)-h(x)}{\ld}&=-iP_u\Big(xh(x)+\frac{1}{2\pi i}(u,f)\Big)=-iT^{\ast}h.
\end{align*}

\noi
Conjugating with $U(t)$, we obtain that
 the restriction of the infinitesimal generator of
$S_{\ld}^{\ast}$ to $\textup{Ran}(H_{u_0})$ is $-iS^{\ast}$.

Moreover, by conjugating $T_{\ld}^{\ast}H_u=H_uT_{\ld}$ with $U(t)$, we obtain
$S_{\ld}^{\ast}H_u=H_uS_{\ld}$. This yields
\begin{align*}
S^{\ast}_{\ld}(\textup{Ran}(H_{u_0}))\subset \textup{Ran}(H_{u_0}).
\end{align*}

\noi
By Theorem X.49, vol. II in \cite{Reed and Simon}, we have that $\textup{Ran}(H_{u_0})$ is a core
of the infinitesimal generator of $S^{\ast}_{\ld}$. Then, the infinitesimal generator of
$S^{\ast}_{\ld}$ is the closure $-iA$ of $-iS^{\ast}$.
\end{proof}

\begin{proof}[Proof of Theorem \ref{general case}]
According to Proposition \ref{prop:xu}, we have that $u(t)\in H^s$ and $xu(t,x)\in L^{\infty}(\R)$
for all $t\in\R$.
Then, by Lemma \ref{lim u}, we obtain that
\begin{align*}
u(t)=\lim_{\eps\to 0}H_{u(t)}(\frac{1}{1-i\eps x}) \text{ in } L^2_+.
\end{align*}

\noi
By Plancherel's identity, this is equivalent to
\begin{align*}
\lim_{\eps\to 0}\mathcal{F}\Big(u(t)\big(1-\frac{1}{1+i\eps x}\big)\Big)= 0 \text{ in } L^2(\R_+).
\end{align*}

\noi
Since,
\begin{align*}
\ft{u}(t,\ld)&=\int_{\R}e^{-i\ld x}u(t,x)dx=\int_{\R}e^{-i\ld x}
u(t)\big(1-\frac{1}{1+i\eps x}\big)dx+\int_{\R}e^{-i\ld x}u(t)\frac{1}{1+i\eps x}dx\\
&=\mathcal{F}\Big(u(t)\big(1-\frac{1}{1+i\eps x}\big)\Big)(\ld)+\int_{\R}e^{-i\ld x}u(t)\frac{1}{1-i\eps x}dx.
\end{align*}

\noi
we obtain that
\begin{align*}
\ft{u}(t,\ld)&=\lim_{\eps\to 0}\Big(u(t),e^{i\ld x}\frac{1}{1-i\eps x}\Big)dx.
\end{align*}

The rest of the proof follows the same lines as the proof of Theorem \ref{thm:general formula},
but uses $T^{\ast}$ and $S^{\ast}$ instead of $T$ and $S$. Special attention should be given to
the fact that the infinitesimal generator of the semi-group $S_{\ld}^{\ast}$ is not $-iS^{\ast}$,
but its closure $-iA$.
\end{proof}

\section{Soliton resolution in the case of strongly generic, rational function data}

\bigskip

We prove that all the solutions
with strongly generic, rational function initial data $u_0\in\M(N)_{\textup{sgen}}$ resolve
into N solitons and a remainder which tends to zero in all the $H^s$-norms for $s\geq 0$,
when $t\to\pm\infty$.

\begin{proof}[Proof of Theorem \ref{thm:soliton emergence}]
The strategy is to write all the vectors
in $\text{Ran} (H_{u_0})$ in the basis $\{e_j\}_{j=1}^N$ and make formula
\eqref{eqn:formula for u} more explicit.

According to Theorem \ref{thm:general formula}, we have
\begin{align*}
S(t)e_j=\Big(\frac{\ld_j^2\nu_j^2}{2\pi}t+(S(0)e_j,e_j)\Big)e_j+\sum_{i=1, i\neq j}^N a_{ji}(t)e_i.
\end{align*}

\noi
Since $a_{ji}(t)$ are linear combinations of
$e^{\pm i\frac{t}{2}(\ld_j^2-\ld_i^2)}$
with constant coefficients, there exists $M>0$ such that
\begin{align*}
|a_{ji}(t)|\leq M,
\end{align*}

\noi
for all $j\neq i$ and all $t\in\R$. Denoting $A_j=\frac{\ld_j^2\nu_j^2}{2\pi}t+(S(0)e_j,e_j)$, the operator
$S$ in the basis $\{e_j\}_{j=1}^N$ can be written as the following matrix:

\[S=\left(
\begin{matrix}
A_1 & a_{12} & \cdots & a_{1N} \\
a_{21} & A_2 & \cdots & a_{2N}  \\
\vdots & \vdots  & \ddots & \vdots \\
a_{N1} & a_{N2} & \cdots & A_N
\end{matrix}
\right)\]

\noi
Let us first compute $\textup{Im}(A_j)=\textup{Im}\big(S(0)e_j,e_j\big)$ for $t$ large enough.
By equation \eqref{eq:S-Sast} and noting that $\tilde{e}(0)=g_0$, we have that
\begin{align*}
2i\textup{Im}\big(S(0)e_j,e_j\big)&=\big(S(0)e_j,e_j\big)-\big(e_j,S(0)e_j\big)=\big((S(0)-S(0)^{\ast})e_j,e_j\big)\\
&=\big(\frac{-1}{2\pi i}(e_j,g_0)g_0,e_j\big)=\frac{i}{2\pi}|(g_0,e_j)|^2.
\end{align*}

\noi
Therefore
\begin{align}\label{eqn:Imc_j(0)}
\textup{Im}\big(S(0)e_j,e_j\big)=\frac{\nu_j^2}{4\pi}.
\end{align}

Then, we notice that
\begin{align}\label{elem integral}
\int_{-\infty}^{\infty}\frac{dx}{|x-at+ib|^2|x-ct+id|^2}=O(\frac{1}{t^2}) \text{ as } t\to\pm\infty,
\end{align}

\noi
if $0<a<c$ and $b,d\neq 0$. This can be proved by estimating the integral on each of the intervals
$(-\infty, at-1], [at-1,at+1], [at+1,ct-1], [ct-1,ct+1], [ct+1, \infty)$ if $t>0$
 large enough, and similarly for $t<0$.
Since $\textup{Im}(A_j)=\frac{\nu_j^2}{4\pi}>0$ and by the strong genericity hypothesis
$\frac{\ld_j^2\nu_j^2}{2\pi}\neq \frac{\ld_k^2\nu_k^2}{2\pi}$ for $j\neq k$, this yields that
\begin{align*}
\frac{1}{(x-A_j)(x-A_k)}=O(\frac{1}{|t|}) \text{ in } L^2(\R) \text{ as } t\to\pm\infty,
\end{align*}

\noi
Moreover, using $\big\|\frac{1}{x-A_j}\big\|_{L^{\infty}}=\frac{1}{\textup{Im}A_j}=\frac{4\pi}{\nu_j^2}$, we have that
$\frac{1}{(x-A_j)(x-A_k)}=O(\frac{1}{t})$ in $H^s(\R)$ for all $s\geq 0$.
Furthermore, we have
\begin{align*}
\Big\|\frac{1}{(x-A_j)(x-A_k)}\Big\|_{L^{\infty}}&=\Big\|\frac{1}{A_k-A_j}\Big(\frac{1}{(x-A_j)}-\frac{1}{(x-A_k)}\Big)\Big\|_{L^{\infty}}\\
&\leq \frac{1}{|A_k-A_j|}\Big(\Big\|\frac{1}{x-A_j}\Big\|_{L^{\infty}}+\Big\|\frac{1}{x-A_j}\Big\|_{\infty}\Big)\\
&=\frac{1}{|A_k-A_j|}
\Big(\frac{4\pi}{\nu_j^2}+\frac{4\pi}{\nu_k^2}\Big)=O(\frac{1}{t}).
\end{align*}

\noi
Therefore,
$\frac{\det (S-xI)}{(A_1-x)\dots (A_N-x)}-1\to 0$ in $L^{\infty}(\R)$ and in $H^s$, $s\geq 0$, as $t\to\pm\infty$,
since it is equal to a linear combination of
$\frac{1}{(x-A_j)(x-A_k)}$,...,$\frac{1}{(x-A_1)\dots(x-A_N)}$. We notice that, using the definition of the determinant, the terms $\frac{1}{x-A_j}$
do not appear in the above linear combination.

Then,
\begin{align*}
(S-xI)^{-1}&=\frac{1}{\det (S-xI)}\left(
\begin{matrix}
C_{11} & C_{12} & \cdots & C_{1N} \\
C_{21} & C_{22} & \cdots & C_{2N}  \\
\vdots & \vdots  & \ddots & \vdots \\
C_{N1} & C_{N2} & \cdots & C_{NN}
\end{matrix}\right)\\
&=\frac{(A_1-x)\dots (A_N-x)}{\det (S-xI)}\left(
\begin{matrix}
\frac{C_{11}}{(A_1-x)\dots (A_N-x)} & \frac{C_{12}}{(A_1-x)\dots (A_N-x)} & \cdots & \frac{C_{1N}}{(A_1-x)\dots (A_N-x)} \\
\frac{C_{21}}{(A_1-x)\dots (A_N-x)} & \frac{C_{22}}{(A_1-x)\dots (A_N-x)} & \cdots & \frac{C_{2N}}{(A_1-x)\dots (A_N-x)}  \\
\vdots & \vdots  & \ddots & \vdots \\
\frac{C_{N1}}{(A_1-x)\dots (A_N-x)} & \frac{C_{N2}}{(A_1-x)\dots (A_N-x)} & \cdots & \frac{C_{NN}}{(A_1-x)\dots (A_N-x)}
\end{matrix}\right),\\
\end{align*}

\noi
where $C_{jj}$ is the cofactor of $A_j-x$, equal to the sum of
$(A_1-x)\dots(A_{j-1}-x)(A_{j+1}-x)\dots (A_N-x)$ and a linear combination of terms
containing at most $N-2$ factors $(A_j-x)$, and $C_{ij}$, $i\neq j$ is the cofactor of $a_{ij}$,
equal to a linear combination of terms
containing at most $N-2$ factors $(A_j-x)$.
Then, we have
\begin{align*}
(S-xI)^{-1}=\left(
\begin{matrix}
\frac{1}{A_1-x}+O(\frac{1}{t}) & O(\frac{1}{t}) & \cdots & O(\frac{1}{t}) \\
O(\frac{1}{t}) & \frac{1}{A_2-x}+O(\frac{1}{t}) & \cdots & O(\frac{1}{t})  \\
\vdots & \vdots  & \ddots & \vdots \\
O(\frac{1}{t}) & O(\frac{1}{t}) & \cdots & \frac{1}{A_N-x}+O(\frac{1}{t})
\end{matrix}
\right) \text{ as } t\to\pm\infty \text{ in } H^s(\R).
\end{align*}

\noi
Therefore,
\begin{align*}
W(S-xI)^{-1}Wg_0&=W(S-xI)^{-1}(\beta_1e^{i\frac{t}{2}\ld_1^2},\dots,\beta_Ne^{i\frac{t}{2}\ld_N^2})^t \\
&=\Big(\frac{e^{it\ld_1^2}\beta_1}{x-A_1}+O(\frac{1}{t}),\dots,\frac{e^{it\ld_N^2}\beta_N}{x-A_N}+O(\frac{1}{t})\Big)^t.
\end{align*}

\noi
Since $u_0=\sum_{j=1}^N(u_0,e_j)e_j$ and by \eqref{sym H_u},
\begin{equation}\label{eqn:(u0,ej)}
(u_0,e_j)=(H_{u_0}g_0,e_j)=(H_{u_0}e_j,g_0)=\ld_j \cj{(g_0,e_j)}=\ld_j\cj{\beta_j},
\end{equation}

\noi
we have
\begin{align*}
u(t)=\frac{1}{2\pi}\big(u_0, W(S-xI)^{-1}Wg_0\big)&=\frac{1}{2\pi}\cdot\frac{e^{-it\ld_1^2}\ld_1\cj{\beta}^2_1}{x-\bar{A}_1}
+\dots+\frac{1}{2\pi}\cdot\frac{e^{-it\ld_N^2}\ld_N\cj{\beta}^2_N}{x-\bar{A}_N}+O(\frac{1}{t})
\end{align*}

\noi
in $H^{s}$, $s\geq 0$. Since $\text{Im}(\bar{A}_j)=-\frac{\nu_j^2}{4\pi}<0$, we have that $u\in H^{s}_+$. Moreover, by \eqref{soliton}, we have that
each of the functions $\frac{1}{2\pi}\cdot\frac{e^{-it\ld_j^2}\ld_1\cj{\beta}^2_j}{x-\bar{A}_j}$
is a soliton of speed $c=\frac{\ld_j^2\nu_j^2}{2\pi}$ and frequency $\omega=\ld_j^2$.
\end{proof}

Let us notice that the result in Theorem \ref{thm:soliton emergence}
 can also be restated in terms of $N$-solitons.

\begin{definition}
A $N$-soliton is a solution of the Szeg\"o equation $u(t)$, such that there exist $N$
solitons $\frac{C_1(t)}{x-q_1(t)},\dots,\frac{C_N(t)}{x-q_N(t)}$ satisfying
\begin{align*}
&\bigg\|u(t)-\sum_{j=1}^N \frac{C_j(t)}{x-q_j(t)}\bigg\|_{H^{1/2}_+}\to 0 \text{ as } t\to -\infty.
\end{align*}

\noi
If, moreover, there exist $\delta_j\in\R$, $j=1,2,\dots,N$ such that
\begin{align*}
&\bigg\|u(t)-\sum_{j=1}^N \frac{C_j(t)}{x-\delta_j-q_j(t)}\bigg\|_{H^{1/2}_+}\to 0 \text{ as } t\to +\infty,
\end{align*}

\noi
we say that the $N$-soliton is pure or that the collision of the $N$ solitons
$\frac{C_j(t)}{x-p_j(t)}$ is elastic, in the sense that there
is no loss of energy in the collision.
\end{definition}

Theorem \ref{thm:soliton emergence} states for $s=1/2$ that if $u_0\in\M(N)_{\textup{sgen}}$,
then the corresponding solution is a pure $N$-soliton. Moreover, there is no
shift in the trajectories of the N solitons,
i.e. $\delta_j=0$ for all $j=1,2,\dots,N$. This situation is characteristic
to completely integrable equations. For the one dimensional cubic NLS,
KdV and mKdV, which are all completely integrable,
 it is known that $N$-solitons exist and are pure \cite{Zakharov, Hirota}.
For the gKdV equation with fourth order nonlinearity, which is not completely integrable,
it was proved in \cite{Martel Merle} that the collision of solitons fails to be
elastic by loss of a small quantity of energy.

\bigskip

\section{Asymptotic behavior of the solution in the case of non-generic, rational function data}

\bigskip

We show that
when $u_0\in\M(2)$ is such that $H_{u_0}^2$ has a double eigenvalue,
then the solution $u$ behaves as the sum of a soliton and a remainder,
which tends to zero in the $H^s$-norms, $0\leq s<1/2$. However,
$\|u(t)\|_{H^s}\to \infty$ if $s>1/2$.
An example of such an initial condition is $u_0=\frac{2}{x+i}-\frac{4}{x+2i}$.
 The operator $H_{u_0}^2$
has the double eigenvalue $(\frac{1}{3})^2$ in this case.

Let us consider an orthonormal basis $\{\tilde{e}_1,\tilde{e}_2\}$ of
$\textup{Ran}(H_{u_0})$ such that $H_{u_0}\tilde{e}_j=\ld\tilde{e}_j$.
Denoting $\tilde{\beta}_j=(g_0,\tilde{e}_j)$ and $\tilde{\nu}_j=|\tilde{\beta}_j|$
we have $g_0=\tilde{\beta}_1\tilde{e}_1+\tilde{\beta}_2\tilde{e}_2$
 and $\|g_0\|_{L^2}^2=\tilde{\nu}_1^2+\tilde{\nu}_2^2$.
By \eqref{eqn:cj{beta_j}beta_i in R}, we have that $\cj{\tilde{\beta}}_1\tilde{\beta}_2\in\R$. We assume
that $\cj{\tilde{\beta}}_1\tilde{\beta}_2=\cj{\tilde{\nu}}_1\tilde{\nu}_2$, and thus
$\tilde{\beta}_j=e^{i\theta}\tilde{\nu}_j$ for $j=1,2$.

We make the following change of basis
\begin{align*}
e_1:=&\frac{1}{\|g_0\|_{L^2}}(\tilde{\nu}_1\tilde{e}_1+\tilde{\nu}_2\tilde{e}_2),\\
e_2:=&\frac{1}{\|g_0\|_{L^2}}(\tilde{\nu}_2\tilde{e}_1-\tilde{\nu}_1\tilde{e}_2).
\end{align*}

\noi
Notice that this is also an orthonormal basis of $\textup{Ran}(H_{u_0})$ and $H_{u_0}e_j=\ld e_j$.
Moreover, setting
$\beta_j:=(g_0,e_j)$ and $\nu_j=|\beta_j|$, we have
\begin{align}\label{beta2=0}
\beta_2:=(g_0,e_2)=0,
\end{align}

\noi
and $\nu_2:=|\beta_2|=0$. In the case when $\cj{\tilde{\beta}}_1\tilde{\beta}_2=-\cj{\tilde{\nu}}_1\tilde{\nu}_2$,
we can similarly choose an orthonormal basis for which $\beta_2=\nu_2=0$.

\begin{lemma}\label{lemma F(t)}
With the notations in
Theorem \ref{thm:conterexample} we set
$c_j(0)=(S(0)e_1,e_j)$, $d_j(0)=(S(0)e_2,e_j)$ for $j=1,2$ and
\begin{align}\label{B}
&A:=\frac{\ld^2\nu_1^2}{2\pi},\\
&B:=\frac{\ld^2\nu_1^2}{\pi}(c_1(0)-d_2(0)),\notag\\
&C:=(c_1(0)-d_2(0))^2+4c_2(0)d_1(0).\notag
\end{align}

Then, $4A^2C-B^2>0$ and $\textup{Im}(B)=\frac{\ld^2\nu_1^4}{4\pi^2}>0$.
\end{lemma}

\begin{proof}
By equation \eqref{eqn:Imc_j(0)} we have that
$\textup{Im}\,c_1(0)=\frac{\nu_1^2}{4\pi}$
and $\textup{Im}\,d_2(0)=\frac{\nu_2^2}{4\pi}=0$.
Then, we obtain
\begin{align}\label{Im B}
\textup{Im}(B)=\frac{\ld^2\nu_1^2}{\pi}\textup{Im}c_1(0)=\frac{\ld^2\nu_1^4}{4\pi^2}.
\end{align}

\noi
Let us notice that
\begin{align}\label{eqn:beta_j}
\beta_j=2\pi i\cj{\Ld (e_j)}.
\end{align}

\noi
Indeed, since $e_j\in\text{Ran}(H_{u_0})$, there exists $f_j\in L^2_+$
such that $e_j=H_{u_0}(f_j)$ and by equation \eqref{eqn:LdH} we have
\begin{align*}
\Ld (e_j)=\Ld (H_{u_0}(f_j))=-\frac{1}{2\pi i}(u_0,f_j)=-\frac{1}{2\pi i}(H_{u_0}g_0,f_j)
=-\frac{1}{2\pi i}(H_{u_0}f_j,g_0)=-\frac{1}{2\pi i}\cj{\beta}_j.
\end{align*}

\noi
Then,
\begin{align*}
4A^2C-B^2=&4\Big(\frac{\ld^2\nu_1^2}{2\pi}\Big)^2\Big((c_1(0)-d_2(0))^2+4c_2(0)d_1(0)\Big)-4\Big(\frac{\ld^2\nu_1^2}{2\pi}\Big)^2(c_1(0)-d_2(0))^2\\
=&16\Big(\frac{\ld^2\nu_1^2}{2\pi}\Big)^2c_2(0)d_1(0).
\end{align*}

\noi
By equation \eqref{eq:S-Sast} and noticing that $\tilde{e}(0)=g_0$, we have
\begin{align*}
d_1(0)=(S(0)e_2,e_1)=(S^{\ast}(0)e_2,e_1)-\frac{1}{2\pi i}(e_2,g_0)(g_0,e_1)=(S^{\ast}(0)e_2,e_1)=(e_2,S(0)e_1)=\cj{c_2(0)}.
\end{align*}

\noi
Thus,
\begin{align*}
4A^2C-B^2=16\Big(\frac{\ld^2\nu_1^2}{2\pi}\Big)^2|d_2(0)|^2.
\end{align*}

Suppose by absurd that $d_2(0)=(S(0)e_2,e_1)=0$. Since $(e_2,e_1)=0$,
and since $e_1,e_2, S(0)e_2$ belong to the two dimensional complex
vector space $\textup{Ran}(H_{u_0})$, it results that there exists $a\in\C$ such that
$S(0)e_2=ae_2$. Using the fact that $S(0)=T$ and the definition of $T$, we obtain that
$e_2(x-a)=\Ld(e_2)b_{u_0}$. Then, by equation \eqref{beta2=0} and \eqref{eqn:beta_j}, we obtain that
$\Ld(e_2)=0$ and therefore $e_2=0$, which is a contradiction. Hence, $4A^2C-B^2>0$.
\end{proof}

\begin{proof}[Proof of Theorem \ref{thm:conterexample}]
Let us first express $S$ in the basis $\{e_1,e_2\}$ of $\text{Ran}(H_{u_0})$.
By Theorem \ref{thm:general formula} we have
\begin{align*}
S(t)e_1=c_1(t)e_1+c_2(t)e_2,
\end{align*}

\noi
with $c_1(t)=\frac{\ld^2\nu_1^2}{2\pi}t+c_1(0)$ and $c_2(t)=\frac{\ld^2}{2\pi}\cj{\beta}_1\beta_2t+c_2(0)=c_2(0)$, and
\begin{align*}
S(t)e_2=d_1(t)e_1+d_2(t)e_2,
\end{align*}

\noi
with $d_1(t)=\frac{\ld^2}{2\pi}\cj{\beta}_1\beta_2t+d_1(0)=d_1(0)$
and $d_2(t)=\frac{\ld^2\nu_2^2}{2\pi}t+d_2(0)=d_2(0)$. We denoted
$c_j(0)=\big(S(0)e_1,e_j\big)$ and $d_j(0)=\big(S(0)e_2,e_j)$, $j=1,2$.
Moreover, by equation \eqref{eqn:cj{beta_j}beta_i in R}, we have that $\cj{\beta}_1\beta_2\in\R$.

Therefore, in the basis $\{e_1,e_2\}$, the operator $S$ is the matrix
\[S=\left(
\begin{matrix}
c_1 & d_1 \\
c_2 & d_2
\end{matrix}
\right).\]

\noi
and its characteristic equation is $x^2-(c_1+d_2)+c_1d_2-d_1c_2=0$.
Since $(\bar{\beta}_1\beta_2)^2=\nu_1^2\nu_2^2$, we obtain that the discriminant of this equation
writes
\begin{align}\label{delta}
\Delta=&(c_1-d_2)^2+4d_1c_2=\Big(\frac{\ld^2 \nu_1^2}{2\pi}t+c_1(0)-d_2(0)\Big)^2
+4c_2(0)d_1(0)\\
=&\Big(\frac{\ld^2\nu_1^2}{2\pi}\Big)^2
t^2+\frac{\ld^2\nu_1^2}{\pi}(c_1(0)-d_2(0))t+\big(c_1(0)-d_2(0)\big)^2+4c_2(0)d_1(0)\notag\\
=&A^2t^2+Bt+C,\notag
\end{align}

\noi
where $A,B,C$ are defined in \eqref{B}.
\noi
The eigenvalues of $S$ will be written in terms of $\sqrt{\Delta}$,
where we use the principal determination of the square root. In order to do so,
we have to make sure that $\Delta$ is not negative. We will show that when $|t|$
is large enough, $\Delta$ cannot be a real number.
In what follows we suppose that $t>0$. The case $t<0$ can be treated similarly.

Using equations \eqref{eqn:Imc_j(0)} and $(\cj{\beta}_1\beta_2)^2=\nu_1^2\nu_2^2$, we obtain
\begin{align*}
\text{Im}(\Delta)&=\frac{\ld^2\nu_1^2}{\pi}\big(\text{Im}(c_1(0))-\text{Im}(d_2(0))\big)t
+\text{Im}\Big(\big(c_1(0)-d_2(0)\big)^2+4c_2(0)d_1(0)\Big)\\
&=\frac{\ld^2\nu_1^4}{4\pi^2}t
+\text{Im}\Big(\big(c_1(0)-d_2(0)\Big)^2+4c_2(0)d_1(0)\Big)
\end{align*}

\noi
and thus $\text{Im}(\Delta)\neq 0$ for $|t|$ large enough. Using the Taylor approximation
$(1+x)^{1/2}=1+\frac{x}{2}-\frac{x^2}{8}+\frac{x^3}{16}+x^3\eps(x)$ if $|x|<1$, we have by \eqref{delta} that
\begin{align*}
\sqrt{\Delta}=&At\Big(1+\frac{B}{A^2t}+\frac{C}{A^2t^2}\Big)^{1/2}\\
=&At\Big(1+\frac{B}{2A^2t}+\frac{C}{2A^2t^2}-\frac{1}{8}\big(\frac{B}{A^2t}+\frac{C}{A^2t^2}\big)^2
+\frac{1}{16}\big(\frac{B}{A^2t}+\frac{C}{A^2t^2}\big)^3+\big(\frac{B}{A^2t}+\frac{C}{A^2t^2}\big)^3\eps \big(\frac{B}{A^2t}+\frac{C}{A^2t^2}\big)\Big)\\
=&At\Big(1+\frac{B}{2A^2t}+\frac{1}{t^2}\Big(\frac{C}{2A^2}-\frac{B^2}{8A^4}\Big)+\frac{1}{t^3}\Big(-\frac{BC}{4A^4}+\frac{B^3}{16A^6}\Big)+O(\frac{1}{t^4})\Big)\\
=&At+\frac{B}{2A}+\frac{4A^2C-B^2}{8A^3}\cdot\frac{1}{t}+\frac{B(B^2-4A^2C)}{16A^6}\cdot\frac{1}{t^2}+O(\frac{1}{t^3})
\end{align*}

\noi
We set
\begin{align}\label{F}
F(t):=\frac{4A^2C-B^2}{8A^3}\cdot\frac{1}{t}+\frac{B(B^2-4A^2C)}{16A^6}\cdot\frac{1}{t^2}+O(\frac{1}{t^3}).
\end{align}

\noi
By Lemma \ref{lemma F(t)}, we have that
\begin{align}
|F(t)|=&\frac{4A^2C-B^2}{8A^3}\cdot\frac{1}{t}+O(\frac{1}{t^2})\label{|F|},\\
\textup{Im}\,F(t)=&-\frac{\ld^2\nu_1^4}{4\pi^2}\cdot\frac{(4A^2C-B^2)}{16A^6}\cdot\frac{1}{t^2}+O(\frac{1}{t^3})\label{ImF}
\end{align}

\noi
with $4A^2C-B^2>0$. Then, we have
\begin{align}\label{racine delta}
\sqrt{\Delta}=At+\frac{B}{2A}+F(t)=At+c_1(0)-d_2(0)+F(t).
\end{align}

\noi
and the eigenvalues of $S$ are
\begin{align}
E_1=&\frac{c_1+d_2+\sqrt{\Delta}}{2}=\frac{\ld^2\nu_1^2}{2\pi}t
+c_1(0)+\frac{F(t)}{2}\label{eqn:E1}\\
E_2=&\frac{c_1+d_2-\sqrt{\Delta}}{2}=d_2(0)-\frac{F(t)}{2}.\label{eqn:E2}
\end{align}

\noi
Therefore,
\begin{align*}
(S-xI)^{-1}=\frac{1}{\det (S-xI)}\left(
\begin{matrix}
d_2-x & -d_1 \\
-c_2 & c_1-x
\end{matrix}
\right)
=\frac{1}{(x-E_1)(x-E_2)}\left(
\begin{matrix}
d_2-x & -d_1 \\
-c_2 & c_1-x
\end{matrix}
\right)
\end{align*}

\noi
and
\begin{align*}
(S-xI)^{-1}Wg_0=(S-xI)^{-1}(e^{i\frac{t}{2}\ld^2}\beta_1,0)^t=\Big(e^{i\frac{t}{2}\ld^2}\frac{\big(d_2(0)-x\big)\beta_1}{(x-E_1)(x-E_2)},
-e^{i\frac{t}{2}\ld^2}\frac{c_2(0)\beta_1}{(x-E_1)(x-E_2)}e_2\Big).
\end{align*}

\noi
Since $u_0=\ld\bar{\beta}_1e_1+\ld\bar{\beta}_2e_2=\ld\bar{\beta}_1e_1$, we obtain
\begin{align}\label{interim u}
u(t)=&\frac{1}{2\pi}\big(u_0,W(S-xI)^{-1}Wg_0\big)=\frac{\ld e^{-it\ld^2}}{2\pi}\cdot\frac{\Big(\cj{d_2(0)}-x\Big)\cj{\beta}^2_1}{(x-\bar{E}_1)(x-\bar{E}_2)}.
\end{align}

\noi
Using \eqref{eqn:E2}, we obtain that
\begin{align*}
u(t)=-\frac{\frac{\ld}{2\pi}e^{-it\ld^2}\cj{\beta}_1^2}{x-\cj{E}_1}
+\bar{F}(t)\frac{\frac{\ld}{4\pi}e^{-it\ld^2}\cj{\beta}_1^2}{(x-\bar{E}_1)(x-\bar{E}_2)}.
\end{align*}

\noi
Let us denote
\begin{align*}
R(t,x)=\bar{F}(t)\frac{\frac{\ld}{4\pi}e^{-it\ld^2}\cj{\beta}_1^2}{(x-\bar{E}_1)(x-\bar{E}_2)}.
\end{align*}

\noi
We will study the $H^s$-norms of $R$, for $s\geq 0$ .
First, we determine $\text{Im}(E_1)$ and $\text{Im}(E_2)$.
By equations \eqref{eqn:E2}, \eqref{eqn:Imc_j(0)}, and \eqref{beta2=0}, we have
\begin{align*}
\text{Im}(E_2)=\textup{Im}(d_2(0))-\frac{\textup{Im}F(t)}{2}=
\frac{\nu_2^2}{4\pi}-\frac{\textup{Im}F(t)}{2}=-\frac{\textup{Im}F(t)}{2}.
\end{align*}

\noi
and similarly, we obtain that
\begin{align*}
\text{Im}(E_1)=\frac{\nu_1^2}{4\pi}+\frac{\textup{Im}F(t)}{2}.
\end{align*}

\noi
Let us now estimate $\|R(t,x)\|_{H^s}$. First, we write
\begin{align}\label{R}
R(t,x)=\frac{\bar{F}(t)}{\bar{E}_1-\bar{E}_2}\cdot\frac{\ld}{4\pi}e^{-it\ld^2}
\cj{\beta}_1^2\big(\frac{1}{x-\bar{E}_1}-\frac{1}{x-\bar{E}_2}\Big).
\end{align}

\noi
We compute the $\dot{H}^s$-norm, $s\geq 0$, of each of the two terms in $R$.
Let $p\in\C$, $\textup{Im}p<0$. By \eqref{residue thm}, we have that
\begin{align*}
\Big\|\frac{1}{x-p}\Big\|_{\dot{H}^s}^2=\int_{0}^{\infty}\xi^{2s}\big|\mathcal{F}\Big(\frac{1}{x-p}\Big)(\xi)\Big|^2d\xi
=c\int_0^{\infty}\xi^{2s}|e^{-ip\xi}|^2d\xi=c\int_0^{\infty}\xi^{2s}e^{2\textup{Im}(p)\xi}d\xi.
\end{align*}

\noi
Integrating by parts, we can explicitly compute the last integral.
If $p=\bar{E}_2$, then $\textup{Im}(p)=\textup{Im}(\bar{E_2})$ and we obtain
\begin{align*}
\Big\|\frac{1}{x-\bar{E}_2}\Big\|_{\dot{H}^s}=O\Big(\frac{1}{|\textup{Im}(\bar{E}_2)|^{(2s+1)/2}}\Big)
=O\Big(\frac{1}{|\textup{Im}F(t)|^{(2s+1)/2}}\Big).
\end{align*}

\noi
More precisely, by \eqref{ImF}  there exist $c,C>0$ such that
\begin{align*}
c|t|^{2s+1}\leq \Big\|\frac{1}{x-\bar{E}_2}\Big\|_{\dot{H}^s}\leq C|t|^{2s+1}.
\end{align*}

\noi
Similarly, for $p=\bar{E}_1$, we have $\textup{Im}(p)=\textup{Im}(\bar{E}_1)
=-\frac{\nu_1^2+\nu_2^2}{4\pi}-\frac{\textup{Im}F(t)}{2}$ and thus
\begin{align*}
\Big\|\frac{1}{x-\bar{E}_1}\Big\|_{\dot{H}^s}=O(1).
\end{align*}

\noi
Consequently, by \eqref{R}, \eqref{eqn:E1}, \eqref{eqn:E2}, \eqref{|F|}, \eqref{ImF}
 we obtain for $0\leq s<\frac{1}{2}$ that
\begin{align*}
\|R(t,x)\|_{H^s}&\leq C\frac{|F(t)|}{|E_1-E_2|}\Big(\Big\|\frac{1}{x-\bar{E}_1}\Big\|_{L^2}
+\Big\|\frac{1}{x-\bar{E}_2}\Big\|_{L^2}\Big)+
C\frac{|F(t)|}{|E_1-E_2|}\Big(\Big\|\frac{1}{x-\bar{E}_1}\Big\|_{\dot{H}^s}
+\Big\|\frac{1}{x-\bar{E}_2}\Big\|_{\dot{H}^s}\Big)\\
&\leq \frac{C}{|t|^2}\big(|t|+|t|^{2s+1}\big).
\end{align*}

\noi
and thus $\|R(t,x)\|_{H^s}\to 0$ for $0\leq s<\frac{1}{2}$.
 For $s>\frac{1}{2}$ we have that
\begin{align*}
\|R(t,x)\|_{\dot{H}^s}\geq C\frac{|F(t)|}{|E_1-E_2|}
\Big(\Big\|\frac{1}{x-\bar{E}_2}\Big\|_{H^s}-\Big\|\frac{1}{x-\bar{E}_1}\Big\|_{H^s}\Big)
\geq \frac{C}{|t|^2}\big(|t|^{2s+1}-|t|\big).
\end{align*}

\noi
Therefore, $\|R(t,x)\|_{H^s}\to+\infty$ if $s>\frac{1}{2}$.

Moreover, for $s=1/2$ we have
\begin{align*}
&c\frac{|F(t)|}{|E_1-E_2|}
\Big(\Big\|\frac{1}{x-\bar{E}_2}\Big\|_{L^2}-\Big\|\frac{1}{x-\bar{E}_1}\Big\|_{L^2}+
\Big\|\frac{1}{x-\bar{E}_2}\Big\|_{\dot{H}^{1/2}}-\Big\|\frac{1}{x-\bar{E}_1}\Big\|_{\dot{H}^{1/2}}\Big)\\
&\leq \|R(t,x)\|_{H^{1/2}}\leq C\frac{|F(t)|}{|E_1-E_2|}\Big(\Big\|\frac{1}{x-\bar{E}_1}\Big\|_{L^2}+\Big\|\frac{1}{x-\bar{E}_2}\Big\|_{L^2}+
\Big\|\frac{1}{x-\bar{E}_1}\Big\|_{\dot{H}^{1/2}}+\Big\|\frac{1}{x-\bar{E}_2}\Big\|_{\dot{H}^{1/2}}\Big).
\end{align*}

\noi
and thus there exist $0<c\leq C$ such that
\begin{align*}
&c\leq \|R(t,x)\|_{H^{1/2}}\leq C
\end{align*}

\noi
for $|t|$ large enough. We proceed similarly for $\|R(t,x)\|_{L^{\infty}}$.
\begin{align*}
&c\frac{|F(t)|}{|E_1-E_2|}
\Big(\Big\|\frac{1}{x-\bar{E}_2}\Big\|_{L^{\infty}}-\Big\|\frac{1}{x-\bar{E}_1}\Big\|_{L^{\infty}}\Big)
\leq \|R(t,x)\|_{L^{\infty}}\\
&\leq C\frac{|F(t)|}{|E_1-E_2|}\Big(\Big\|\frac{1}{x-\bar{E}_1}
\Big\|_{L^{\infty}}+\Big\|\frac{1}{x-\bar{E}_2}\Big\|_{L^{\infty}}\Big).
\end{align*}

\noi
Since $\big\|\frac{1}{x-\bar{E}_j}\big\|_{L^{\infty}}=\frac{1}{|\textup{Im}E_j|}$ for $j=1,2$,
there exist $0<c\leq C$ such that
\begin{align*}
c<c\frac{1}{t^2}(t^2-1)
\leq \|R(t,x)\|_{L^{\infty}}\leq C\frac{1}{t^2}(t^2+1)<C.
\end{align*}

\noi
Hence, $R(t,x)$ stays away from zero in the $H^{1/2}$-norm and $L^{\infty}$-norm.

Setting
\begin{align*}
p(t):=\frac{\ld^2\nu_1^2}{2\pi}t
+\text{Re}(c_1(0))-i\frac{\nu_1^2}{4\pi}\\
\end{align*}

\noi
we have $\bar{E}_1(t)=p(t)+O(\frac{1}{t})$ as $t\to\pm\infty$ and
\begin{align}\label{u non-generic}
u(t,x)=&-\frac{\frac{\ld}{2\pi}\cj{\beta}_1^2e^{-it\ld^2}}{x-p(t)}
+R(t,x)+\tilde{\eps}(t,x).
\end{align}

\noi
where $\tilde{\eps}(t,x)=-\frac{\ld}{2\pi}\cj{\beta}_1^2e^{-it\ld^2}\Big(\frac{1}{x-\bar{E}_1}-\frac{1}{x-p(t)}\Big)$
and
\begin{align*}
\tilde{\eps}(t,x)=C\frac{\bar{E}_1-p(t)}{(x-\bar{E}_1)(x-p(t))}=O(\frac{1}{t}) \text{ in all } H^s, s\geq 0.
\end{align*}

By \eqref{soliton}, the first term in the sum in \eqref{u non-generic}
is a soliton.
Using equation \eqref{beta2=0}, we have that $\|u_0\|_{L^2}^2=(u_0,u_0)=(H_{u_0}g_0,H_{u_0}g_0)=\ld^2\nu_1^2$.
In \cite[Lemma 3.5]{pocov} it was shown that $H_{u_0}$ is a Hilbert-Schmidt operator of
Hilbert-Schmidt norm $\frac{\|u_0\|_{\dot{H}^{1/2}}}{\sqrt{2\pi}}$. Then,
$2\ld^2=\textup{Tr}(H_{u_0}^2)=\frac{\|u_0\|^2_{\dot{H}^{1/2}}}{2\pi}$. Therefore, the soliton satisfies
\begin{align*}
\big|\frac{\ld}{2\pi}\cj{\beta}_1^2e^{-it\ld^2}\Big|&=\frac{\ld\nu_1^2}{2\pi}
=\frac{\ld^2\nu_1^2}{2\pi\ld}=\frac{\|u_0\|_{L^2}^2}{\sqrt{\pi}\|u_0\|_{\dot{H}^{1/2}}},\\
\textup{Im}(p)&=-\frac{\nu_1^2}{4\pi}=-\frac{\ld^2\nu_1^2}{\ld^2}=-\frac{\|u_0\|_{L^2}^2}{\|u_0\|_{\dot{H}^{1/2}}^2}.
\end{align*}

\noi
We set $\eps(t,x)=R(t,x)+\tilde{\eps}(t,x)$. Then, $\eps(t,x)\to 0$ as
$t\to\pm\infty$ in all the $H^s$-norms, $0\leq s<1/2$.
However, $\lim_{t\to\infty}\|\eps(t,x)\|_{H^s}=\infty$
if $s>1/2$ and $t\to\pm\infty$.
\end{proof}

\begin{proof}[Proof of Corollary \ref{Corollary}]
Notice that the Sobolev norms of solitons are constant in time.
Then, the solution in Theorem \ref{thm:conterexample},
having a non-generic initial data $u_0\in\M(2)$ such that $H_{u_0}$ has a double eigenvalue,
provides an example of a solution whose
$H^s$-norms, with $s>1/2$ grow
\[\|u(t)\|_{H^s}\geq C|t|^{2s-1} \text{ if } s>1/2\]

\noi
and $|t|$ is big enough.

This does not contradict the complete integrability of the Szeg\"o equation,
since the conservation laws $J_{2n}=(u,H_u^{2n-2}(u))$ can all be controlled by the
$H^{1/2}_+$-norm, as it was noticed in Remark \ref{conservation laws}.
\end{proof}

\section{Generalized action-angle coordinates}

On $L^2_+(\R)$ we introduce the symplectic form
\begin{equation*}
\omega(u,v)=4 \textup{Im}\int_{\R}u\bar{v}.
\end{equation*}

\noi
A function $F:L^2_+(\R)\to\R$ admits a Hamiltonian vector field $X_F$ if
\begin{equation*}
d_uF(h)=\omega(h,X_{F}(u)),
\end{equation*}

\noi
for all $u,h\in L^2_+(\R)$. If the functions $F,G:L^2_+(\R)\to\R$ admit
the Hamiltonian vector fields $X_F,X_G$, then we define the Poisson bracket of
$F$ and $G$ by:
\begin{equation*}
\{F,G\}(u)=\omega(X_F(u),X_G(u))=d_uG(X_F(u)).
\end{equation*}

A consequence of the Lax pair is the existence of an
infinite sequence of conservation laws as we noticed in Corollary \ref{cor:conserv}.

We now introduce the Szeg\"o hierarchy, i.e. the evolution equations
associated to the Hamiltonian vector fields of $J_{2n}$ for all $n\in\N^{\ast}$, and prove that each of
these equations possesses a Lax pair. We will need the following lemma:

\begin{lemma}
For all $f\in L^2(\R)$ we have
\begin{equation}\label{eqn basic}
(I-\Pi )f=\cj{\Pi (\bar{f})}.
\end{equation}

\noi
As a consequence, the following identity holds:
\begin{equation}
H_{aH_u(a)}(h)=H_u(a)H_a(h)+H_u(a\Pi (\bar{a}h)).
\end{equation}

\end{lemma}

\begin{proof}
The first equation is equivalent to $f=\Pi (f)+\cj{\Pi (\bar{f})}$ and it follows
by passing into the Fourier space. Then
\begin{align*}
H_{aH_u(a)}(h)&=\Pi (aH_u(a)\bar{h})=\Pi \Big(H_u(a)\big(\Pi (a\bar{h})+(I-\Pi)(a\bar{h})\big)\Big)\\
&=H_u(a)H_a(h)+\Pi (H_u(a)\cj{\Pi(\bar{a}h)})
=H_u(a)H_a(h)+\Pi (u\bar{a}\cj{\Pi(\bar{a}h)})\\
&=H_u(a)H_a(h)+H_u(a\Pi(\bar{a}h)).
\end{align*}
\end{proof}

\begin{proposition}
Let $u\in H^s_+$, $s>\frac{1}{2}$. The Hamiltonian vector field associated to $J_{2n}(u)$ is
\begin{equation}
X_{J_{2n}}(u)=\frac{1}{2i}\sum_{k=0}^{n-1}H_u^{2n-2k-1}(g)H_u^{2k}(g)
\end{equation}

\noi
Moreover,
\begin{equation}\label{eqn: H_X_J_2n}
H_{X_{J_{2n}}}(u)=[B_{u,n},H_u],
\end{equation}

\noi
where
\begin{equation}
B_{u,n}(h)=-\frac{i}{4}\sum_{j=0}^{2n-2}H_u^{j}(g)\Pi (\cj{H_u^{2n-2-j}(g)}h).
\end{equation}

\end{proposition}

\begin{proof}
The proof follows using the above lemma and similar computations as in the proof of Theorem 8.1 in \cite{PGSG}. Denote
\begin{align*}
&w(x):=(1-xH_u^2)^{-1}(u)=\sum_{n=0}^{\infty}x^nH_u^{2n}u\\
&J(x,u):=(u,w(x,u))=\sum_{n=0}^{\infty}x^nJ_{2n+2}(u).
\end{align*}

\noi
A computation shows that
\begin{align*}
&d_uJ(x,u)(h)=\omega (h,X(x)), \text{ where }\\
&X(x)=\frac{1}{2i}\big(w(x)+xw(x)H_uw(x)\big).
\end{align*}

\noi
Identifying the coefficients of $x^n$, we obtain the desired formula for $X_{J_{2n}}(u)$.
For the second part of the proposition, we use
\[w(x)=u+xH_u^2(w)\]

\noi
and the above lemma to obtain
\begin{align*}
H_{iX_{J(x,u)}}(h)&=\frac{1}{2}H_{w}(h)+\frac{x}{2}H_{wH_u(w)}(h)=\frac{1}{2}H_{w}(h)
+\frac{x}{2}H_{uH_u(w)}(h)+\frac{x^2}{2}H_{H_u^2(w)H_u(w)}(h)\\
&=G_uH_u(h)+H_uD_u(h),
\end{align*}

\noi
where
\begin{align*}
G_u(h)&=\frac{x}{2}w\Pi (\bar{w}h)\\
D_u(h)&=\frac{1}{2}\sum_{n=0}^{\infty}x^nhH_u^{2n-1}(u)+\frac{x}{2}H_hH_u(w)+\frac{x^2}{2}H_u(w)\Pi (\cj{H_u(w)})
\end{align*}

\noi
Since, by \eqref{sym H_u}, $H_{iX_{J(x,u)}}(h)=\frac{1}{2}H_{w}(h)+\frac{x}{2}H_{wH_u(w)}(h)$ satisfies
\begin{align*}
(H_{iX_{J(x,u)}}(h_1),h_2)=(H_{iX_{J(x,u)}}(h_2),h_1)
\end{align*}

\noi
for all $h_1,h_2\in L^2_+$, we have that
\begin{align*}
H_{iX_{J(x,u)}}(h)=G_uH_u(h)+H_uD_u(h)=H_uG_u(h)+D_uH_u(h)=C_uH_u+H_uC_u,
\end{align*}

\noi
where $C_u=\frac{G_u+D_u}{2}$. Identifying once more the coefficients of $x^n$ and using
the fact that $H_u$ is a skew-symmetric operator, we obtain the formula for $H_{X_{J_{2n}}}(u)$.
\end{proof}

As in \cite{PGSG}, the following result holds:
\begin{theorem}\label{thm: complete vector fields}
For every $u_0\in H^s_+$, $s>1$, there exists a unique solution $u\in C(\R,H^s_+)$
of the Cauchy problem
\begin{equation}
\begin{cases}\label{eqn:X_J_2n}
\dt u=X_{J_{2n}}(u)\\
u(0)=u_0.
\end{cases}
\end{equation}

\noi
Moreover, $u$ satisfies
\begin{align}\label{ dt H_u}
\dt H_u=[B_{u,n},H_u]
\end{align}

\noi
and
\begin{align}\label{eqn: involution}
\{J_{2n},J_{2k}\}=0,
\end{align}

\noi
for all $k\in\N^{\ast}$.
\end{theorem}

In what follows we compute $g'(t)$ and the commutator $[T,B_{u,n}]$
and use this result to determine the evolution of the
angles and generalized angles along the flow of $X_{J_{2n}}$.

\begin{lemma}
Let $u\in C(\R,H^s_+)$, $s>1$ be a solution of \eqref{eqn:X_J_2n} and for all $t\in\R$ let
$g(t)\in \textup{Ran} (H_{u(t)})$ be such that $H_{u(t)}g(t)=u(t)$. Then
\begin{align}\label{eqn: g'}
g'(t)=\frac{i}{4}H_u^{2n-2}(g)+P_uB_{u,n}(g).
\end{align}

\noi
Moreover,
\begin{align*}
[T,B_{u,n}]h=\frac{1}{8\pi}\sum_{j=1}^{2n-2}H_u^j(g)\big(h,H_u^{2n-2-j}(g)\big)+\frac{1}{8\pi} \big(h,H_u^{2n-2}(g)\big)g-\Ld (h)B_{u,n}(g).
\end{align*}

\end{lemma}

\begin{proof}
In order to compute $g'(t)$, we differentiate with respect to time the equality $H_u(g)=u$:
\begin{align*}
[B_{u,n},H_u]g+H_u(g')=X_{J_{2n}}(u).
\end{align*}

\noi
Thus
\begin{align*}
H_u(g')=&X_{J_{2n}}(u)-[B_{u,n},H_u]g\\
=&-\frac{i}{2}\sum_{k=0}^{n-1}H_u^{2n-2k-1}(g)H_u^{2k}(g)+\frac{i}{4}\sum_{j=0}^{2n-2}H_u^{j}(g)\Pi (\cj{H_u^{2n-2-j}(g)}u)+H_uB_{u,n}(g)\\
=&-\frac{i}{2}\sum_{k=0}^{n-1}H_u^{2n-2k-1}(g)H_u^{2k}(g)+\frac{i}{4}\sum_{j=0}^{n-1}H_u^{2j}(g)H_u^{2n-1-2j}(g)\\
&+\frac{i}{4}\sum_{j=1}^{n-1}H_u^{2n-2j-1}(g)H_u^{2j}(g)+H_uB_{u,n}(g)\\
=&-\frac{i}{4}H_u^{2n-1}(g)+H_uB_{u,n}(g).
\end{align*}

\noi
Using the fact that $H_u$ is a skew-symmetric operator and is onto on its range, we obtain \eqref{eqn: g'}.

Since the product of two rational functions has $\Ld$ equal to zero, we notice that
\begin{align*}
\Ld (B_{u,n}h)=-\frac{i}{4}\sum_{j=0}^{2n-2}\Ld \Big(H_u^j(g)\Pi \Big(\cj{H_u^{2n-2-j}(g)}h\Big)\Big)=-\frac{i}{4}\Ld \Big(\Pi \big(\cj{H_u^{2n-2}(g)}h\big)\Big)
\end{align*}

\noi
A similar computation yields the second equation in the statement.
\end{proof}

\begin{proposition}\label{prop: phi_j,c_j}
If $u_0\in H^s_+$, $s>1$ and $u_0\in\M(N)_\textup{gen}$, then the solution $u(t)$ of the
equation \eqref{eqn:X_J_2n} is contained in the toroidal cylinder $TC(u_0)$ defined by \eqref{TC(u)},
for all $t\in\R$. Moreover, the angles $\phi_j$ and the generalized angles
$\gamma_j$ evolve along the flow of this equation as follows:
\begin{align*}
&\{J_{2n},\phi_j\}=\frac{d}{dt}\phi_j=\frac{\ld_j^{2n-2}}{4}\\
&\{J_{2n},\gamma_j\}=\frac{d}{dt}\gamma_j=\frac{n-1}{4\pi}\ld_j^{2n-2}\nu_j^2.
\end{align*}

\end{proposition}

\begin{proof}
Since the evolution equation \eqref{eqn:X_J_2n} admits the Lax pair \eqref{ dt H_u},
the classical theory yields that if $u(t)$ is a solution of \eqref{eqn:X_J_2n}, then
\begin{align*}
H_{u(0)}=U_n(t)^{\ast}H_{u(t)}U_n(t),
\end{align*}

\noi
where $U_n(t)$ is a unitary operator on $L^2_+$ satisfying
\begin{align*}
\frac{d}{dt}U_n(t)=B_{u,n}U_n,\,\,\,\,\,\,\,\,U(0)=I.
\end{align*}

\noi
Therefore, the eigenvalues $\ld_j^2$, $j=1,2,\dots,N$ of $H^2_{u(t)}$ are conserved in time.
Moreover, if we denote by $\{e_j(0)\}_{j=0}^N$ an orthonormal basis of $\textup{Ran}(H_{u_0})$
such that $H_{u_0}e_j(0)=\ld_j e_j(0)$, then $e_j(t)=U_n(t)e_j(0)$ form a basis of
$\textup{Ran} (H_{u(t)})$ such that $H_{u(t)}e_j(t)=\ld_j e_j(t)$.
Then, by \eqref{eqn: g'} and using the fact that $B_{u,n}$ is a skew-symmetric operator, we have
\begin{align*}
\frac{d}{dt}\big(e_j(t),g(t)\big)&=\big(B_{u,n}e_j(t),g(t)\big)-\frac{i}{4}\big(e_j(t),H_u^{2n-2}(g(t))\big)
+\big(e_j(t),B_{u,n}(g(t))\big)\\
&=-\frac{i}{4}\ld_j^{2n-2}\big(e_j(t),g(t)\big).
\end{align*}

\noi
Therefore $\big(e_j(t),g(t)\big)=e^{-i\phi_j(t)}\big(e_j(0),g_0\big)$, with $\frac{d}{dt}\phi_j=\frac{\ld_j^{2n-2}}{4}$.
Thus $\big|\big(e_j(t),g(t)\big)\big|=\big|\big(e_j(0),g_0\big)\big|$ and $u(t)\in TC(u_0)$ for all $t\in\R$.

By Lemma \ref{lemma: Ld(H_uf)} and equation \eqref{eqn:eps}, we have that
$\Ld (H_uh)=-\frac{1}{2\pi i}(u,h)=-\frac{1}{2\pi i}\lim_{\eps\to 0}\Big(H_uh,\frac{1}{1-i\eps x}\Big)$. Then
\begin{align*}
\frac{d}{dt}\gamma_j(t)=&\frac{d}{dt}\Big(T(t)e_j(t),e_j(t)\Big)
=\lim_{\eps\to 0}\frac{d}{dt}\Big(xe_j(t)-\frac{1}{2\pi i}\big(e_j(t),\frac{1}{1-i\eps x}\big)(g(t)-1),e_j(t)\Big)\\
=&\lim_{\eps\to 0}\Big(xB_{u,n}e_j(t)-\frac{1}{2\pi i}\big(B_{u,n}e_j(t),\frac{1}{1-i\eps x}\big)(g(t)-1),e_j(t)\Big)\\
&-\frac{1}{2\pi i}\lim_{\eps\to 0}(e_j(t),\frac{1}{1-i\eps x})\Big(\frac{i}{4}H_{u}^{2n-2}(g)+B_{u,n}(g),e_j(t)\Big)+\Big(T(t)e_j(t),B_{u,n}e_j(t)\Big)\\
&=\Big([T,B_{u,n}]e_j(t),e_j(t)\Big)+\Ld(e_j(t))\Big(\frac{i}{4}H_u^{2n-2}(g)+B_{u,n}(g),e_j(t)\Big)\\
&=\frac{1}{8\pi}\sum_{k=1}^{2n-2}\Big(e_j(t),H_u^{2n-2-k}(g)\Big)\Big(H_u^k(g),e_j(t)\Big)+\frac{1}{8\pi}\Big(e_j(t),H_u^{2n-2}(g)\Big)\big(g,e_j(t)\big)\\
&-\Ld (e_j(t))\Big(B_{u,n}(g),e_j(t)\Big)+\frac{i}{4}\Ld (e_j(t))\Big(H_u^{2n-2}(g),e_j(t)\Big)+\Ld (e_j(t))\Big(B_{u,n}(g),e_j(t)\Big).
\end{align*}

\noi
Writing $e_j(t)=H_{u(t)}f_j(t)\in\textup{Ran} (H_{u(t)})$, we have
\begin{align*}
\Ld (e_j(t))&=\Ld \big(H_{u(t)}f_j(t)\big)=-\frac{1}{2\pi i}\big(u(t),f_j(t)\big)=-\frac{1}{2\pi i}\big(H_{u(t)}g(t),f_j(t)\big)\\
&=-\frac{1}{2\pi i}\big(H_{u(t)}f_j(t),g(t)\big)
=-\frac{1}{2\pi i}\big(e_j(t),g(t)\big)\\
&=-\frac{1}{2\pi i}e^{-i\phi_j(t)}\big(e_j(0),g_0\big)=-\frac{1}{2\pi i}\big(e_j(t),g(t)\big).
\end{align*}

\noi
Then
\begin{align*}
\frac{d}{dt}\gamma_j(t)=\frac{1}{8\pi}\sum_{k=1}^{2n-2}\ld_j^{2n-2}\nu_j^2+\frac{1}{8\pi}\ld_j^{2n-2}\nu_j^2
-\frac{1}{8\pi}\ld_j^{2n-2}\nu_j^2=\frac{n-1}{4\pi}\ld_j^{2n-2}\nu_j^2.
\end{align*}
\end{proof}

\begin{proposition}\label{prop:one to one}
If $u\in\M(N)_{\textup{gen}}$, then
\begin{align*}
u(x)=\frac{i}{2\pi}\sum_{j,k=1}^N\ld_j\nu_j\nu_k e^{2i\phi_j}\cj{(T-xI)_{jk}^{-1}},
\end{align*}

\noi
where

\begin{align}\label{eqn:T}
Te_j=\sum_{k\neq j}\frac{\ld_j\nu_j\nu_k}{2\pi i}\cdot\frac{\ld_j
-\ld_k e^{i(2\phi_j-2\phi_k)}}{\ld_k^2-\ld_j^2}e_k+(\gamma_j+i\frac{\nu_j^2}{4\pi})e_j,
\end{align}

\noi
for all $j\in\{1,2,\dots,N\}$. In particular, $\chi$ is a one to one map.

\end{proposition}

\begin{proof}
The proof follows the same lines as the proof of Theorem \ref{thm:general formula}.
The only difference is that we work with the orthonormal basis $\tilde{e}_j=e^{i\phi_j}e_j$.
Since $H_u$ is anti-linear,
the orhonormal basis $\{e_j\}_{j=1}^N$ satisfying $H_ue_j=\ld_je_j$ is determined only
modulo the sign of $e_j$. Therefore, $\phi_j=\arg(e_j,u_0)$ is determined modulo $\pi$.
We intend to introduce generalized action-angle coordinates,
and the angles should be defined
modulo $2\pi$. Considering the basis $\tilde{e}_j$, the formulas we obtain only depend on
$2\phi_j$, which are therefore good candidates for the angles.
\end{proof}

\begin{proof}[Proof of Theorem \ref{thm:action-angle}]

Let us first notice that, if we prove that $\chi$ is a symplectic diffeomorphism, then
the coordinates $(2\ld_j^2\nu_j^2, 4\pi\ld_j^2, 2\phi_j, \gamma_j)$ are canonical.
Denote $I_j=2\ld_j^2\nu_j^2$. By equation $E=2J_4=2\sum_{j=1}^N\ld_j^4\nu_j^2=\sum_{j=1}^N\ld_j^2 I_j$
and using Proposition \ref{prop: phi_j,c_j}, we obtain that for the flow of the Szeg\"o equation we have:
\begin{align*}
&\frac{d}{dt}(2\phi_j(t))=\{E,2\phi_j\}=4\{J_4,\phi_j\}=\ld_j^2\\
&\frac{d}{dt}\gamma_j(t)=\{E,\gamma_j\}=2\{J_4,\gamma_j\}=\frac{\ld_j^2\nu_j^2}{2\pi}.
\end{align*}

\noi
Thus, the Szeg\"o equation can be indeed rewritten as
\begin{align*}
\begin{cases}
&\frac{d}{dt}I_j=0\\
&\frac{d}{dt}\phi_j(t)=\frac{\partial E}{\partial I_j}\\
&\frac{d}{dt}(4\pi\ld_j^2)=0\\
&\frac{d}{dt}\gamma_j(t)=\frac{\partial E}{\partial (4\pi\ld_j^2)}.
\end{cases}
\end{align*}

The first step in proving that $\chi$ is a symplectic diffeomorphism
is to compute the Poisson brackets between actions and (generalized) angles. This will lead to
$\chi$ being a local diffeomorphism.

\subsection{Poisson brackets between actions and (generalized) angles}

First notice that
\begin{equation}\label{J_2n}
J_{2n}(u)=(u,H_u^{2n-2}u)=\sum_{k=1}^N\ld_k^{2n}\nu_k^2.
\end{equation}

\noi
Fix $j\in\{1,2,\dots,N\}$. Writing
\begin{align*}
\{J_{2n},2\phi_j\}=\sum_{k=1}^N\ld_k^{2n-2}\{\ld_k^2\nu_k^2,2\phi_j\}+\sum_{k=1}^N(n-1)\ld_k^{2n-2}\nu_k^2\{\ld_k^2,2\phi_j\}
\end{align*}

\noi
for all $n=1,2,\dots,2N$ we obtain the following linear system of equations:
\begin{align*}
\sum_{k=1}^N\ld_k^{2n-2}\{\ld_k^2\nu_k^2,2\phi_j\}+\sum_{k=1}^N(n-1)\ld_k^{2n-2}\nu_k^2\{\ld_k^2,2\phi_j\}=\frac{\ld_j^{2n-2}}{2}
\end{align*}

\noi
with $2N$ unknowns, $\{\ld_k^2\nu_k^2,\phi_j\}$ and $\ld_k^2\nu_k^2\{\ld_k^2,\phi_j\}$.
The matrix of this system is invertible. Indeed, supposing by absurd that the matrix is
 not invertible, it results that the columns are linearly dependent. Therefore, there
exist numbers $a_n$, $n=1,2,\dots,2N$, not all zero, such that
\begin{align*}
&\sum_{n=1}^{2N}a_n(\ld_k^2)^{n-1}=0,\,\,\,\,\,\,\,\,\,\,\,\,\,\,\,\,\,\,\,\sum_{n=1}^{2N}a_n(n-1)(\ld_k^2)^{n-2}=0.
\end{align*}

\noi
Considering the polynomial $P(x)=\sum_{n=1}^{2N}a_nX^{n-1}$, this yields
$P(\ld_k^2)=0$, $P'(\ld_k^2)=0$. Thus, each $\ld_k^2$ is a double root of the polynomial.
 Since a polynomial of degree $2N-1$ with $2N$ roots is identically zero, we obtain a contradiction.
Therefore, the system is a Cramer system and one can easily verify that its solutions are
\begin{align}
&\{\ld_k^2\nu_k^2,2\phi_j\}=\frac{\delta_{kj}}{2}\label{1}\\
&\{\ld_k^2,2\phi_j\}=0\label{2}.
\end{align}

\noi
Similarly, computing $\{J_{2n},\gamma_j\}$, we obtain the Cramer system
\begin{align*}
\sum_{k=1}^N\ld_k^{2n-2}\{\ld_k^2\nu_k^2,\gamma_j\}+\sum_{k=1}^N(n-1)\ld_k^{2n-2}\nu_k^2\{\ld_k^2,\gamma_j\}
=\frac{n-1}{4}\ld_j^{2n-2}\nu_j^2
\end{align*}

\noi
with solutions
\begin{align}
&\{\ld_k^2\nu_k^2,\gamma_j\}=0\label{3}\\
&\{\ld_k^2,\gamma_j\}=\frac{\delta_{kj}}{4\pi}\label{4}.
\end{align}

\noi
Since $\ld_j^2$ and $\nu_j^2$ are conserved by the flow of any of the
equations in the Szeg\"o hierarchy, we have that
\begin{align*}
&\{J_{2n},\ld^2_j\}=\frac{d}{dt}\ld_j^2=0,\,\,\,\,\,\,\,\,\,\,\{J_{2n},\nu^2_j\}=\frac{d}{dt}\nu_j^2=0.
\end{align*}

\noi
Proceeding as above, we have two homogeneous Cramer systems,
whose solutions must be null. Thus, we obtain
\begin{align}
&\{\ld_k^2\nu_k^2,\ld^2_j\}=0\label{5}\\
&\{\ld_k^2,\ld^2_j\}=0\label{6}\\
&\{\ld_k^2\nu_k^2,\ld^2_j\nu_j^2\}=0\label{7}.
\end{align}

\subsection{$\chi$ is a local diffeomorphism}

The fact that $\chi$ is a local diffeomorphism
is equivalent to proving that the differentials $d\ld_j^2$, $d(\ld_j^2\nu_j^2)$, $d\phi_j$,
$d\gamma_j$, $j=1,2,\dots,N$, are linearly independent. Suppose
\begin{align*}
\sum_{j=1}^N\alpha _jd(\ld_j^2)+\beta _j d(\ld_j^2\nu_j^2)+\theta _jd\phi_j+\eta _jd\gamma_j=0.
\end{align*}

\noi
Applying this differential to the vector field $X_{\ld_k^2}$, using $df(X_g)=\{g,f\}$,
\eqref{6}, \eqref{5}, \eqref{2}, and \eqref{4}, we obtain
\begin{align*}
\sum_{j=1}^N\eta _j\frac{\delta_{kj}}{4\pi}=0
\end{align*}

\noi
and thus $\eta_k=0$, for all $k=1,2,\dots,N$.
Applying the same differential to $X_{\ld_k^2\nu_k^2}$ and using \eqref{5}, \eqref{7}, and
\eqref{1}, we obtain $\theta_k=0$ for all $k=1,2,\dots,N$. Applying the differential to
 $X_{\phi_k}$ and using \eqref{2} and \eqref{1} we have $\beta_k=0$ for all $k=1,2,\dots,N$.
Finally, applying the differential to $X_{c_k}$ and using \eqref{4} we obtain $\alpha_k=0$
for all $k=1,2,\dots,N$.
Therefore, $d\ld_j^2$, $d(\ld_j^2\nu_j^2)$, $d\phi_j$, $d\gamma_j$, $j=1,2,\dots,N$,
are linearly independent and $\chi$ is a local diffeomorphism.

\medskip

Since a bijective local diffeomorphism is a diffeomorphism, and we have by Proposition \ref{prop:one to one}
that $\chi$ is one to one, we only need to show that
$\chi$ is onto. A proper local diffeomorphism taking values in a connected manifold is onto.
Thus, it is enough to show that $\chi$ is proper.

\subsection{$\chi$ is a proper mapping}

Let $K\subset\Omega$ be a compact set. Set
\begin{align*}
&\big(I^{(p)},\tilde{I}^{(p)},2\phi^{(p)},\gamma^{(p)}\big):=\Big(2\big(\ld_j^{(p)}\big)^2\big(\nu_j^{(p)}\big)^2,4\pi\big(\ld_j^{(p)}\big)^2,2\phi_j^{(p)},\gamma_j^{(p)}\Big)_{j=1}^N\\
&(I,\tilde{I},2\phi,\gamma):=\Big(2\ld_j^2\nu_j^2, 4\pi\ld_j^2,2\phi_j,\gamma_j\Big)_{j=1}^N.
\end{align*}

Let
$(I^{(p)},\tilde{I}^{(p)},2\phi^{(p)},\gamma^{(p)}), (I,\tilde{I},2\phi,\gamma)\in K$ such that
\begin{equation}\label{convergent sequence}
\Big(I_j^{(p)},\tilde{I}_j^{(p)},2\phi_j^{(p)},\gamma_j^{(p)}\Big)\to(I,\tilde{I},2\phi,\gamma) \text{ as } p\to\infty.
\end{equation}

Consider $u_p\in\chi^{-1}\Big(I^{(p)},\tilde{I}^{(p)},2\phi^{(p)},\gamma^{(p)}\Big)$.
Then $\big(\ld_j^{(p)}\big)^2$ are the eigenvalues of $H_{u_p}^2$. By Lemma 3.5 in \cite{pocov}, which states that $H_u$ is a Hilbert-Schmidt operator
of norm $\|H_u\|_{H-S}=\frac{1}{\sqrt{2\pi}}\|u\|_{\dot{H}^{1/2}_+}$, we have that
\begin{align*}
\|u_p\|_{L^2_+}^2=&J_2(u_p)=\sum_{j=1}^N\big(\ld_j^{(p)}\big)^2(\nu_j^{(p)})^2=\frac{1}{2}\sum_{j=1}^N I_j^{(p)},\\
\|u_p\|_{\dot{H}^{1/2}_+}^2=&2\pi\|H_{u_p}\|_{H-S}^2=2\pi\textup{Tr}(H_{u_p}^2)=2\pi\sum_{j=1}^N\big(\ld_j^{(p)}\big)^2=\frac{1}{2}\sum_{j=1}^N\tilde{I}_j^{(p)}.
\end{align*}

\noi
Since $I^{(p)}\to I$ and $\tilde{I}^{(p)}\to \tilde{I}$ as $p\to\infty$, this yields that $\|u_p\|_{H^{1/2}}$ is bounded.
Consequently, there exists $u\in H^{1/2}_+$ such that $u_p\wk u$ in $H^{1/2}_+$. It follows in particular that
$u_p\to u$ in $L^2_{\textup{loc}}$. We denote by $\ld_j(u)$, $\nu_j(u)$, $\phi_j(u)$, and $\gamma_j(u)$ the
spectral data for $u$.

By Proposition \ref{prop:one to one}, we have that
\begin{align*}
u_p(x)=\frac{i}{2\pi}\sum_{j,k=1}^N\ld^{(p)}_j\nu^{(p)}_j\nu^{(p)}_k e^{2i\phi^{(p)}_j}\cj{(T^{(p)}-xI)_{jk}^{-1}},
\end{align*}

\noi
where $(T^{(p)}-xI)_{jk}^{-1}$ is a component of the matrix $(T^{(p)}-xI)^{-1}$ in the basis $\{e_j^{(p)}\}_{j=1}^N$, and

\begin{align*}
T^{(p)}e^{(p)}_j=\sum_{k\neq j}\frac{\ld^{(p)}_j\nu^{(p)}_j\nu^{(p)}_k}{2\pi i}\cdot\frac{\ld^{(p)}_j
-\ld^{(p)}_k e^{i(2\phi^{(p)}_j-2\phi^{(p)}_k)}}{(\ld^{(p)}_k)^2-(\ld^{(p)}_j)^2}e_k+\Big(\gamma^{(p)}_j+i\frac{\big(\nu^{(p)}_j\big)^2}{4\pi}\Big)e_j,
\end{align*}

\noi
for all $j\in\{1,2,\dots,N\}$. By equation \eqref{convergent sequence}, there exists $R>0$ such that
$\|T^{(p)}\|\leq \frac{R}{2}$ for all $p\in\N$. Using the Neumann series, we have that if $|x|\geq R$, then
there exists $A>0$ such that
\[\|(T^{(p)}-xI)^{-1}\|\leq \frac{A}{|x|},\]

\noi
for all $p\in\N$. This yields that
\begin{align*}
\lim_{R\to\infty}\sup_{p}\int_{|x|>R}|u_p(x)|^2dx\leq \lim_{R\to\infty}\int_{|x|>R}\frac{A^2}{|x|^2}dx=0.
\end{align*}

\noi
Since $u_p\to u$ in $L^2_{\textup{loc}}$, this triggers $u_p\to u$ in $L^2_+(\R)$.

Let us now prove that $H_{u_p}(h)\to H_u(h)$ in $L^2_+$, for all $h\in L^2_+$. First, notice that
there exists $C>0$ such that
\[\|H_{u_p}-H_u\|=\|H_{u_p-u}\|\leq \|H_{u_p-u}\|_{H-S}\leq \frac{1}{\sqrt{2\pi}}\|u_p-u\|_{\dot{H}^{1/2}_+}\leq C.\]

\noi
In particular, it follows that it suffices to prove
that $H_{u_p}(h)\to H_u(h)$ for $h$ in a dense subset of $L^2_+$,
for example $h\in L^{\infty}\cap L^2_+$.
For such $h$, we have that
\[\|H_{u_p}(h)-H_u(h)\|_{L^2_+}=\|H_{u_p-u}(h)\|_{L^2_+}=\|\Pi\big((u_p-u)\bar{h}\big)\|_{L^2}\leq \|u_p-u\|_{L^2_+}\|h\|_{L^{\infty}}\]

\noi
Thus, $u_p\to u$ in $L^2(\R)$ yields $H_{u_p}(h)\to H_u(h)$ in $L^2_+$.

As a consequence, we have that $J_{2n}(u_p)\to J_{2n}(u)$ as $p\to\infty$. Indeed, we write
\begin{align*}
J_{2n}(u_p)-J_{2n}(u)=&(H_{u_p}^{2n-2}u_p,u_p)-(H_u^{2n-2}u,u)\\
=&(H_{u_p}^{2n-2}u_p,u_p-u)+(H_{u_p}^{2n-2}(u_p-u),u)\\
&+\sum_{j=1}^{2n-2}(H_{u_p}^{2n-2-j}H_{u_p-u}H_u^{j-1}u,u).
\end{align*}

\noi
For the first term we notice that
\begin{align*}
|(H_{u_p}^{2n-2}u_p,u_p-u)|&\leq \|H_{u_p}^{2n-2}u_p\|_{L^2}\|u_p-u\|_{L^2}
\leq \|H_{u_p}\|^{2n-2}\|u_p\|_{L^2_+}\|u_p-u\|_{L^2_+}\\
&\leq C\|u_p\|_{\dot{H}^{1/2}_+}^{2n-2}\|u_p\|_{L^2_+}\|u_p-u\|_{L^2_+}\to 0
\text{ as } p\to\infty.
\end{align*}

\noi
For the second term we have that
\begin{align*}
|(H_{u_p}^{2n-2}(u_p-u),u)|&\leq \|H_{u_p}^{2n-2}(u_p-u)\|_{L^2_+}\|u\|_{L^2_+}
\leq \|H_{u_p}\|^{2n-2}\|u_p-u\|_{L^2_+}\|u\|_{L^2_+}\\
&\leq \|u_p\|_{\dot{H}^{1/2}_+}^{2n-2}\|u_p-u\|_{L^2_+}\|u\|_{L^2_+}\to 0
\text{ as } p\to\infty.
\end{align*}

\noi
For the other terms, in the case when $j$ is even, we use the self-adjointness of the operator
$H_u^2$. We then obtain:
\begin{align*}
|(H_{u_p}^{2n-2-j}H_{u_p-u}H_u^{j-1}u,u)|=&|(H_{u_p-u}H_u^{j-1}u,H_{u_p}^{2n-2-j}u)|\\
&\leq
\|H_{u_p-u}H_u^{j-1}u\|_{L^2_+}\|H_{u_p}^{2n-2-j}u\|_{L^2_+}\\
& \leq \|H_{u_p-u}H_u^{j-1}u\|_{L^2_+}\|u_p\|_{\dot{H}^{1/2}_+}^{2n-2-j}\|u\|_{L^2_+},
\end{align*}

\noi
and the first factor tends to zero since $H_{u_p-u}(h)\to 0$ in $L^2_+$ for all $h\in L^2_+$.
For the case when $j$ is odd, we use equation \eqref{sym H_u} and then proceed similarly.

We prove in the following that $\ld_j(u)=\ld_j$ and $\nu_j(u)=\nu_j$.
Since, by equation \eqref{J_2n}, $J_{2(n+1)}(u_p)=\sum_{j=1}^N\big(\ld_j^{(p)}\big)^{2(n+1)}\big(\nu_j^{(p)}\big)^2$,
we have that
\begin{align*}
\sum_{n=0}^{\infty}x^nJ_{2(n+1)}(u_p)&=\sum_{n=0}^{\infty}x^n\sum_{j=1}^N\big(\ld_j^{(p)}\big)^{2(n+1)}\big(\nu_j^{(p)}\big)^2\\
&=\sum_{j=1}^N\big(\ld_j^{(p)}\big)^{2}\big(\nu_j^{(p)}\big)^2\sum_{n=0}^\infty x^n\big(\ld_j^{(p)}\big)^{2n}
=\sum_{j=1}^N\frac{\big(\ld_j^{(p)}\big)^{2}\big(\nu_j^{(p)}\big)^2}{1-x\big(\ld_j^{(p)}\big)^2},
\end{align*}

\noi
for $|x|<1/\ld_j^2$, and thus for every $x$ distinct from the poles.
Then, using $\ld_j^{(p)}\to \ld_j$ and $\nu_j^{(p)}\to \nu_j$, we obtain
\begin{align*}
\sum_{n=0}^{\infty}x^nJ_{2(n+1)}(u_p)\to\sum_{j=1}^N\frac{\ld_j^{2}\nu_j^2}{1-x\ld_j^2}.
\end{align*}

\noi
On the other hand, we have that
\begin{align*}
\sum_{n=0}^{\infty}x^nJ_{2(n+1)}(u_p)
\to\sum_{n=0}^{\infty}x^nJ_{2(n+1)}(u)=
\sum_{j=1}^N\frac{\big(\ld_j(u)\big)^{2}\big(\nu_j(u)\big)^2}{1-x\big(\ld_j(u)\big)^2}.
\end{align*}

\noi
Therefore,
\begin{align*}
\sum_{j=1}^N\frac{\ld_j^{2}\nu_j^2}{1-x\ld_j^2}
=\sum_{j=1}^N\frac{\big(\ld_j(u)\big)^{2}\big(\nu_j(u)\big)^2}{1-x\big(\ld_j(u)\big)^2},
\end{align*}

\noi
which yields, by identification, $\ld_j(u)=\ld_j$ and $\nu_j(u)=\nu_j$.

At last, we show that $e_j(u_p)\to\pm e_j(u)$. It then follows that
\begin{align*}
2\phi_j^{(p)}=&2\phi_j(u_p)=\arg(u_p,e_j(u_p))^2\to \arg(u,e_j(u))^2=2\phi_j(u)\\
\gamma_j^{(p)}=&\gamma_j(u_p)=\textup{Re}(T^{(p)}e_j(u_p),e_j(u_p))\to \textup{Re}(Te_j(u),e_j(u))=\gamma_j(u).
\end{align*}

\noi
Since, by \eqref{convergent sequence}, we also have $2\phi_j^{(p)}\to2\phi_j$ and $\gamma_j^{(p)}\to\gamma_j$, we obtain that $2\phi_j(u)=2\phi_j$ and $\gamma_j(u)=\gamma_j$. Hence $\chi(u)=(I,\tilde{I},2\phi,\gamma)\in K$,
and $u\in \chi^{-1}(K)$. Thus $\chi^{-1}(K)$ is compact, which proves that $\chi$ is proper.

We still need to show that $e_j(u_p)\to\pm e_j(u)$. Using $\ld_j^{(p)}\to\ld_j=\ld_j(u)$,\\
$\nu_j^{(p)}\to\nu_j=\nu_j(u)$, we have that
\[\|u_p\|_{\dot{H}^{1/2}_+}=\sum_{j=1}^N\big(\ld_j^{(p)}\big)^2\big(\nu_j^{(p)}\big)^2\to
\sum_{j=1}^N\big(\ld_j(u)\big)^2\big(\nu_j(u)\big)^2=\|u\|_{\dot{H}^{1/2}_+}.\]

\noi
Since $u_p\wk u$ in $H^{1/2}_+$ and $u\to u$ in $L^2_+$, it follows that $u_p\to u$ in $H^{1/2}_+$.
This yields that $H_{u_p}\to H_u$ in the sense of the norm. As a consequence, setting
\[P_j^{(p)}h:=(h,e_j(u_p))e_j(u_p).\]

\noi
to be the orthogonal projection onto the
eigenspace of $H_{u_p}^2$, corresponding to the eigenvalue $\big(\ld_j^{(p)}\big)^2$
and similarly,
\[P_j(u)h:=(h,e_j(u))e_j(u)\]

\noi
to be the orthogonal projection onto the eigenspace of $H_u^2$, corresponding to the eigenvalue $\ld_j^2(u)$,
we have by Theorem VIII.23 in \cite{Reed and Simon}, that $P_j^{(p)}\to P_j(u)$.

Therefore, $(h,e_j(u_p))e_j(u_p)\to (h,e_j(u))e_j(u)$ as $p\to\infty$, for all $h\in L^2_+$. Taking $h=e_j(u)$,
we have that $(e_j(u),e_j(u_p))e_j(u_p)\to e_j(u)$. Since $e_j(u_p)$ and $e_j(u)$ are unitary vectors,
we notice that $|(e_j(u),e_j(u_p))|=1$. Then, using the relation \eqref{sym H_u}, we have
\begin{align*}
(e_j(u),e_j(u_p))&=\frac{1}{\ld_j^{(p)}}(e_j(u),H_{u_p}e_j(u_p))\\
&=\frac{1}{\ld_j^{(p)}}(e_j(u),H_{u}e_j(u_p))+\frac{1}{\ld_j^{(p)}}(e_j(u),H_{u_p-u}e_j(u_p))\\
&=\frac{1}{\ld_j^{(p)}}(e_j(u_p),H_{u}e_j(u))+\frac{1}{\ld_j^{(p)}}(e_j(u),H_{u_p-u}e_j(u_p))\\
&=\frac{\ld_j(u)}{\ld_j^{(p)}}(e_j(u_p),e_j(u))
+\frac{1}{\ld_j^{(p)}}(e_j(u_p),H_{u_p-u}e_j(u))
\end{align*}

\noi
Letting $p\to\infty$, we obtain
\begin{align*}
\lim_{p\to\infty}(e_j(u),e_j(u_p))=\lim_{p\to\infty}(e_j(u_p),e_j(u))=\cj{\lim_{p\to\infty}(e_j(u),e_j(u_p))}.
\end{align*}

\noi
Since the above limit is of absolute value 1, we obtain that \\$\lim_{p\to\infty}(e_j(u),e_j(u_p))=\pm 1$
and therefore $e_j(u_p)\to \pm e_j(u)$ as $p\to\infty$.

\subsection{$\chi$ is a symplectic transformation}

We proved so far that $\chi$ is a diffeomorphism and we computed the Poisson brackets between
actions and (generalized) angles. In order to prove that $\chi$ is symplectic, we only need to prove that
the Poisson brackets involving only angles and generalized angles,
$\{\phi_j,\phi_k\}$, $\{\gamma_j,\phi_k\}$, and $\{\gamma_j,\gamma_k\}$, are zero.

We first remark that the Jacobi identity
yield that $\{\phi_j,\phi_k\}$, $\{\gamma_j,\phi_k\}$, and $\{\gamma_j,\gamma_k\}$
are only functions of $\ld_{\ell}^2$ and $\ld_{\ell}^2\nu_{\ell}^2$ for $\ell=1,2,\dots,N$.
Indeed, for the first one we take $f=\ld_{\ell}^2$ and then $f=\ld_{\ell}^2\nu_{\ell}^2$ in
\begin{align}
\{f,\{\phi_j,\phi_k\}\}+\{\phi_k,\{f,\phi_j\}\}+\{\phi_j,\{\phi_k,f\}\}=0,
\end{align}

\noi
which gives by \eqref{1}, \eqref{2}, that
\begin{align}
\{\ld_{\ell}^2,\{\phi_j,\phi_k\}\}=\{\ld_{\ell}^2\nu_{\ell}^2,\{\phi_j,\phi_k\}\}=0.
\end{align}

\noi
Writing $\{\phi_j,\phi_k\}=h(\ld_{\ell}^2,\ld_{\ell}^2\nu_{\ell}^2,\phi_{\ell},\gamma_{\ell})$ for $\ell=1,2,\dots,N$, we obtain
\begin{align}
\frac{\partial h}{\partial \phi_{\ell}}=\frac{\partial h}{\partial \gamma_{\ell}}=0.
\end{align}

Define now $J_1(u)=(u,g)$ and $J_3(u)=(H_u^2u,g)$. We will compute
$\{J_1,J_3\}$ to prove that $\{\phi_j,\phi_k\}=0$.
We have that
\begin{align*}
d_uJ_1(h)&=\lim_{t\to 0}\frac{\big(u+th,g(u+th)\big)-(u,g(u))}{t}
=(h,g(u))+\lim_{t\to 0} \Big(u,\frac{g(u+th)-g(u)}{t}\Big)\\
&=(h,g(u))+(u,d_ug(h))=(h,g(u))+\big(H_ug,d_ug(h)\big)=(h,g(u))+\big(H_u(d_ug(h)),g\big).
\end{align*}

\noi
In order to compute $H_u(d_ug(h))$, we differentiate the equation $u=H_ug$:
\begin{align*}
h&=\lim_{t\to 0}\frac{H_{u+th}g(u+th)-H_ug(u)}{t}=\lim_{t\to 0}\Big(H_u\Big (\frac{g(u+th)-g(u)}{t}\Big )+H_hg(u+th)\Big )\\
&=H_u(d_ug(h))+H_h(g).
\end{align*}

\noi
Thus, $H_u(d_ug(h))=h-H_h(g)=\Pi(h(1-\bar{g}))$
and
$d_uJ_1(h)=(h,g)+(h,g(1-g))$.

\noi
Therefore, the vector fields corresponding to the real and imaginary part of $J_1$ are:
\begin{align*}
X_{\text{Re}J_1}&=-\frac{i}{4}\big (g+g(1-g)\big ),\,\,\,\,\,\,X_{\text{Im}J_1}=\frac{1}{4}\big (g+g(1-g)\big ).
\end{align*}

\noi
Similarly we have that
\begin{align*}
d_uJ_3(h)&=(H_u^2h,g)+(H_uH_hu+H_hH_uu,g)+(H_u^2u,d_ug(h))\\
&=(h,H_u^2g)+(H_ug,H_hu)+(H_hH_uu,g)+\big(H_u(d_ug(h)),H_uu\big)\\
&=(h,H_uu)+(u^2,h)+(h,gH_uu)+\big(h,(1-g)H_uu\big)=2(h,H_uu)+(u^2,h).
\end{align*}

\noi
Then
\begin{align*}
\{J_1,J_3\}=&d_uJ_3\cdot X_{\textup{Re}J_1}(u)+id_uJ_3\cdot X_{\textup{Im}J_1}(u)\\
=&2\Big(-\frac{i}{4}\big (g+g(1-g)\big ),H_uu\Big)+\Big(u^2,-\frac{i}{4}\big (g+g(1-g)\big )\Big)\\
&+2i\Big(\frac{1}{4}\big (g+g(1-g)\big ),H_uu\Big)+i\Big(u^2,\frac{1}{4}\big (g+g(1-g)\big )\Big)\\
=&\frac{i}{2}\big(u^2,g+g(1-g)\big).
\end{align*}

\noi
Using equations \eqref{eqn:bar{u}(1-g)} and \eqref{eqn basic}, we have
\begin{align*}
(u^2,g+(1-g)g)=&(u^2,g)+(u^2,(1-g)g)=(u^2,g)+(u(1-\bar{g}),\bar{u}g)\\
=&(u^2,g)+(u(1-\bar{g}),(I-\Pi)(\bar{u}g))=(u^2,g)+(u(1-\bar{g}),\cj{\Pi(u\bar{g})})\\
=&(u^2,g)+(u(1-\bar{g}),\bar{u})
=(u^2,g)+\int_{-\infty}^{\infty}u^2-(u^2,g)=\int_{-\infty}^{\infty}u^2=0.
\end{align*}

\noi
Thus, we obtain $\{J_1,J_3\}=0$.
On the other hand, we have
\begin{align*}
\{J_1,J_3\}=&\Big\{\sum_{j=1}^N\ld_j\nu_j^2e^{-2i\phi_j},\sum_{k=1}^N\ld_k^3\nu_k^2e^{-2i\phi_k}\Big\}\\
=&\sum_{j,k=1}^Ne^{-2i(\phi_j+\phi_k)}\Big(-i\{\ld_j\nu_j^2,\phi_k\}
+i\{\ld_k^3\nu_k^2,\phi_j\}+\{\phi_j,\phi_k\}\ld_j\ld_k^3\nu_j^2\nu_k^2\Big).
\end{align*}

\noi
Since $\{\phi_j,\phi_k\}$
only depends on $\ld_{\ell}^2$ and $\ld_{\ell}^2\nu_{\ell}^2$, $\ell=1,2,\dots,N$,
we have that the coefficient of $e^{-2i(\phi_j+\phi_k)}$ in the above expression is
\[-i\{\ld_j\nu_j^2,\phi_k\}-i\{\ld_k\nu_k^2,\phi_j\}
+i\{\ld_k^3\nu_k^2,\phi_j\}+i\{\ld_j^3\nu_j^2,\phi_k\}
+\{\phi_j,\phi_k\}\ld_j\ld_k\nu_j^2\nu_k^2(\ld_j^2-\ld_k^2).\]

\noi
Comparing the two expressions for $\{J_1,J_3\}$, we have that $\{J_1,J_3\}$ is a trigonometric polynomial which is equal to zero.
Therefore all its coefficients are zero,
which triggers, by taking the real part, that $\{\phi_j,\phi_k\}=0$.

In order to compute $\{\gamma_j,\gamma_k\}$ and $\{2\phi_j,\gamma_k\}$ we denote
$A:=(Tu,u)$, $C:=(Tu,g)$
and compute $\{A,C\}$ in two different ways.
First, we use $\{A,C\}(u)=d_uC\cdot X_{\textup{Re}A}+id_uC\cdot X_{\textup{Im}A}$. Since
\begin{align*}
d_uA(h)=2\textup{Re}(h,Tu)+\Lambda (u)\big(g(1-g),h\big)+\big(h,\frac{1}{2\pi i}(u,g)g\big),
\end{align*}

\noi
for all $h$ rational function (notice that we extend the definition of $T$ to $\bigcup_{N\in\N^{\ast}}\M(N)$),
we obtain the following Hamiltonian vector fields:
\begin{align*}
X_{\textup{Re}A}=&-\frac{i}{2}Tu-\frac{i}{4}\Lambda (u)g(1-g)-\frac{1}{8\pi}(u,g)g,\\
X_{\textup{Im}A}=&-\frac{1}{4}\Lambda (u)g(1-g)-\frac{i}{8\pi}(u,g)g.
\end{align*}

\noi
Similarly we have
\begin{align*}
d_uC(h)=\Lambda (u)(f(1-g),h)-\frac{1}{2\pi i}(u,g)(h,f(1-g))+(Th,g)+(h,(1-g)Tg),
\end{align*}

\noi
where $f$ is the unique element in $\textup{Ran}(H_u)$ such that $H_uf=g$.
By Lemma \ref{lemma:g=1-b}, we have that $\textup{Ker}(H_u)=(1-g)L^2_+$
and using the orthogonality of $\textup{Ker}(H_u)$ and $\textup{Ran}(H_u)$
we obtain
\begin{align*}
\{A,C\}=-\frac{i}{2}(T^2u,g)+\frac{1}{4\pi}\Lambda (u)(u,g)(g,f)-\frac{i}{2}\Lambda (u)(Tg(1-g),g)
-\frac{i}{2}\Lambda (u)(g,Tg).
\end{align*}

\noi
Notice also that, by Lemmas \ref{lemma:g=1-b} and \ref{lemma Tast}, we have
\begin{align*}
Tg(1-g)&=xg(1-g)-\Lambda \big(g(1-g)\big)(1-g)=xg(1-g)-\Lambda (g)(1-g)\\
&=(1-g)T^{\ast}g\in\textup{Ker}(H_u),
\end{align*}

\noi
and thus by orthogonality of $\textup{Ker}(H_u)$ and $\textup{Ran}(H_u)$,
 the third term vanishes. By \eqref{T-Tast},
we rewrite the first term as
\begin{align*}
-\frac{i}{2}(T^2u,g)=-\frac{i}{2}\Big(T^{\ast}Tu-\frac{1}{2\pi i}(Tu,g)g,g\Big)
=-\frac{i}{2}(Tu,Tg)+\frac{1}{4\pi}(Tu,g)(g,g).
\end{align*}

\noi
Hence, we have
\begin{align*}
\{A,C\}=-\frac{i}{2}(Tu,Tg)+\frac{1}{4\pi}(Tu,g)(g,g)+\frac{1}{4\pi}
\Lambda (u)(u,g)(g,f)-\frac{i}{2}\Lambda (u)(g,Tg).
\end{align*}

\noi
Proceeding as in the case of $\{J_1,J_3\}$, we obtain after tedious computations that
$\{\gamma_j,2\phi_k\}=0$ and
$\{\gamma_j,\gamma_k\}=0$
for all $j,k\in\{1,2,\dots,N\}$,
which proves that the coordinates we defined are symplectic.

\end{proof}

\begin{proof}[Proof of Corollary \ref{cor:toroidal cilinders}]
Fixing $\ld_j$ and $\nu_j$ the application $\chi$ in Theorem \ref{thm:action-angle}
yields a diffeomorphism between $TC(u_0)$ and $\T^N\times\R^N$.
\end{proof}

\noi
{\bf Acknowledgments:}
The author is grateful to her Ph.D. advisor Prof. Patrick G\'erard
for introducing her to this subject and for his constant support
during the preparation of this paper.

\end{document}